\documentclass[11pt]{article}

\usepackage{amsmath, amssymb, amsthm} 

\usepackage{dsfont}

\usepackage{tikz}

\allowdisplaybreaks
\expandafter\let\expandafter\oldproof\csname\string\proof\endcsname
\let\oldendproof\endproof
\renewenvironment{proof}[1][\proofname]{%
	\oldproof[\bf #1]%
}{\oldendproof}

\allowdisplaybreaks

\theoremstyle{plain}
\newtheorem{theorem}{Theorem}
\newtheorem{lemma}{Lemma}[section]
\newtheorem{claim}[lemma]{Claim}
\newtheorem{proposition}[lemma]{Proposition}

\newtheorem{corollary}[theorem]{Corollary}
\newtheorem{conjecture}[lemma]{Conjecture}

\newtheorem{remark}[lemma]{Remark}
\newtheorem{definition}[lemma]{Definition}

\RequirePackage[normalem]{ulem} 
\RequirePackage{color}\definecolor{RED}{rgb}{1,0,0}\definecolor{BLUE}{rgb}{0,0,1} 

\title{A New Bound for the Brown--Erd\H{o}s--S\'os Problem}

\author{David Conlon\thanks{Department of Mathematics, California Institute of Technology, Pasadena, CA 91125, USA. Email: dconlon@caltech.edu. Supported in part by NSF Award DMS-2054452.} \and Lior Gishboliner\thanks{Department of Mathematics, ETH, Z\"urich, Switzerland., Email: lior.gishboliner@math.ethz.ch. Supported in part by SNSF grant 200021\_196965.}
	\and Yevgeny Levanzov\thanks{School of Mathematics, Tel Aviv University, Tel Aviv 69978, Israel. Email: yevgenyl@mail.tau.ac.il.
	}
	\and Asaf Shapira\thanks{
		School of Mathematics, Tel Aviv University, Tel Aviv 69978, Israel.
		Email: asafico@tau.ac.il. Supported in part by ISF Grant 1028/16, ERC Consolidator Grant 863438 and NSF-BSF Grant 20196.
}
}

\date{}

\begin{document}

\maketitle

\begin{abstract}

Let $f(n,v,e)$ denote the maximum number of edges in a $3$-uniform hypergraph not containing
$e$ edges spanned by at most $v$ vertices.
One of the most influential open problems in extremal combinatorics
then asks, for a given number of edges $e \geq 3$, what is the smallest integer $d=d(e)$
such that $f(n,e+d,e) = o(n^2)$? This question has its origins in work of
Brown, Erd\H{o}s and S\'os from the early 70's and the standard conjecture is that $d(e)=3$ for every $e \geq 3$.
The state of the art result regarding this problem was obtained in 2004 by S\'{a}rk\"{o}zy and Selkow, who showed that $f(n,e + 2 + \lfloor \log_2 e \rfloor,e) = o(n^2)$.
The only improvement over this result was a recent breakthrough of Solymosi and Solymosi, who improved the bound for $d(10)$ from 5 to 4. We obtain the first asymptotic improvement over the S\'{a}rk\"{o}zy--Selkow bound, showing that
$$
f(n, e + O(\log e/ \log\log e), e) = o(n^2).
$$

\end{abstract}
	
\section{Introduction}\label{sec:intro}

Extremal combinatorics, and extremal graph theory in particular, asks which global properties of a graph
force the appearance of certain local substructures. Perhaps the most well-studied problems of this type are Tur\'an-type
questions, which ask for the minimum number of edges that force the appearance of a fixed subgraph $F$.
Recall that an $r$-uniform hypergraph ($r$-graph for short) is composed of a ground set $V$ of size $n$ (the vertices) and a collection $E$ of subsets of $V$ (the edges), where each edge is of size exactly $r$. Note that an ordinary graph is just a $2$-graph.
A {\em $(v,e)$-configuration} is a hypergraph with $e$ edges and at most $v$ vertices.
Denote by $f_r(n,v,e)$ the largest number of edges in an $r$-graph on $n$ vertices that contains no $(v,e)$-configuration.
Estimating the asymptotic growth of this function for fixed integers $r$, $e$, $v$ and large $n$ is one of the most well-studied and influential problems in extremal graph theory. For example, when $e={v \choose r}$ we get the well-known Tur\'an problem of determining the maximum possible number of edges in an $r$-graph that contains no complete $r$-graph on $v$ vertices. As another example, the case $r=2$, $v=2t$ and $e=t^2$ is essentially equivalent to the Zarankiewicz--K\H{o}v\'{a}ri--S\'os--Tur\'an problem, which asks for the maximum number of edges in a graph without a complete bipartite graph $K_{t,t}$.

Our focus in this paper is on a notorious question of this type, which emerged from work 
of Brown, Erd\H{o}s and S\'os \cite{BHS1,BHS2} in the early 70's and came to be named after them. A special case of this so-called Brown--Erd\H{o}s--S\'os conjecture (see \cite{Erd,EFR}) states the following:


\begin{conjecture}[Brown--Erd\H{o}s--S\'os Conjecture]\label{BESC}
For every $e \geq 3$,
$$
f_3(n,e+3,e)=o(n^2).
$$
\end{conjecture}

Despite much effort by many researchers, Conjecture \ref{BESC} is wide open, having only been settled for $e=3$ by Ruzsa and Szemer\'edi \cite{Ruzsa_Szemeredi} in what is known as the $(6,3)$-theorem. To get some perspective on the significance of this special case of Conjecture \ref{BESC}, suffice it to say that the famous {\em triangle removal lemma} (see \cite{CFS2} for a survey) was devised in order to prove the $(6,3)$-theorem; that \cite{Ruzsa_Szemeredi} was one of the first applications of Szemer\'edi's regularity lemma \cite{Szemeredi}; and that the $(6,3)$-theorem implies Roth's theorem \cite{Roth} on 3-term arithmetic progressions in dense sets of integers. As another indication of the importance of this problem, we note that one of the main driving forces for
proving the celebrated hypergraph removal lemma, obtained by Gowers \cite{Gowers} and R\"odl et al.~\cite{NRS,Rodl_Skokan_1,Rodl_Skokan_2} (see also the paper of Tao \cite{Tao}), was the hope that it would lead to a proof of Conjecture \ref{BESC}.

Since we seem to be quite far\footnote{As an indication of the difficulty of Conjecture \ref{BESC}, let us mention that the case $e=4$ (i.e., the statement $f_3(n,7,4)=o(n^2)$) implies the notoriously difficult Szemer\'edi theorem \cite{SzThm0,SzThm} for $4$-term arithmetic progressions (see~\cite{Erd}).}
from proving Conjecture \ref{BESC}, it is natural to look for approximate versions. Namely, given $e \geq 3$, find the smallest $d=d(e)$ 
such that $f_3(n,e + d,e) = o(n^2)$.
The best result of this type
was obtained 15 years ago by S\'{a}rk\"{o}zy and Selkow \cite{Sarkozy_Selkow}, who proved that
\begin{equation}\label{eq:Sarkozy_Selkow}
f_3\left( n,e + 2 + \lfloor \log_2 e \rfloor, e \right) = o(n^2).
\end{equation}
Since the result of \cite{Sarkozy_Selkow}, the only advance was obtained by Solymosi and Solymosi \cite{Solymosi}, who improved the bound $f_3(n,15,10)=o(n^2)$ that follows from (\ref{eq:Sarkozy_Selkow}) to $f_3(n,14,10)=o(n^2)$.

The main result of this paper, Theorem \ref{thm:main}, gives the first general improvement over \eqref{eq:Sarkozy_Selkow}. Moreover, it shows that one can replace the $\lfloor \log_2 e \rfloor$ ``error term" in \eqref{eq:Sarkozy_Selkow} by a much smaller, sub-logarithmic, \nolinebreak term.

\begin{theorem}\label{thm:main}
For every $e \geq 3$,
$$
f_3\left(n,e+ \lceil 26\log e/\log\log e \rceil , \, e\right) = o(n^2).
$$
\end{theorem}

By using asymptotic estimates for the factorial (in place of cruder bounds), one can replace the multiplicative constant $26$ in the above theorem by $6 + o(1)$.

Although Theorem \ref{thm:main} deals with $3$-graphs, its proof relies on an application of the $r$-graph removal lemma for {\em all} values of $r$. Employing the removal lemma for arbitrary $r$ allows us to overcome a natural barrier which stood in the way of improving the result of \cite{Sarkozy_Selkow}.
Since the proof of Theorem \ref{thm:main} is quite involved, we sketch the main new ideas underlying it in Section \ref{sec:proof_overview}.

As we mentioned above, Conjecture \ref{BESC} has a more general form (see \cite{AS,Sarkozy_Selkow_2}), which states
that for every $2 \leq k < r$ and $e \geq 3$ we have $f_r(n,(r-k)e + k + 1,e)=o(n^k)$. However, it is a folklore
observation that this more general version
is in fact equivalent to the special case stated as Conjecture \ref{BESC} (corresponding to $k=2$ and $r=3$). 
Moreover, approximate results towards Conjecture \ref{BESC} translate to approximate results for the more general conjecture.  
Since this reduction does not
appear in the literature, we prove it here:

\begin{proposition}\label{prop:reduction}
For every $2 \leq k < r$, $e \geq 3$ and $d \geq 1$,
$$
f_r(n,(r-k)e + k + d,e) \leq \left( \binom{r-k+2}{3} \cdot (e-1) + 1 \right) \cdot \frac{\binom{n}{k-2}}{\binom{r}{k-2}} \cdot f_3(n,e + 2 + d,e).
$$	
\end{proposition}

Setting $d = 1$ in the above proposition readily implies that Conjecture \ref{BESC} is indeed equivalent to the general form of the Brown--Erd\H{o}s--S\'os conjecture stated above. The reason for stating the proposition for arbitrary $d$ is that it allows us to infer approximate versions of the general Brown--Erd\H{o}s--S\'os conjecture from approximate versions of Conjecture \ref{BESC}. In particular, by combining Theorem \ref{thm:main} with Proposition \ref{prop:reduction}, we immediately obtain the following corollary.
\begin{corollary}\label{cor:general_BES}
	For every $2 \leq k < r$ and $e \geq 3$,
	$$
	f_r\left(n,(r-k)e + k + \lceil 26\log e/\log\log e \rceil - 2, \, e\right) = o(n^k).
	$$
\end{corollary}

As part of our general discussion of the Brown--Erd\H{o}s--S\'os conjecture, it is also worth mentioning {\em lower bounds} for Conjecture \ref{BESC}. In their seminal paper \cite{Ruzsa_Szemeredi}, where they proved that $f_3(n,6,3) = o(n^2)$, Ruzsa and Szemer\'edi also showed that $f_3(n,6,3) \geq n^{2 - o(1)}$. Their construction relies on Behrend's example \cite{Behrend} of a dense set of integers with no 3-term arithmetic progressions. A natural question is then whether $f_3(n,e+3,e) \geq n^{2-o(1)}$ for every $e \geq 3$. A simple averaging argument shows that every $(7,4)$-configuration or $(8,5)$-configuration contains a (6,3)-configuration and, hence, $f_3(n,7,4),f_3(n,8,5) \geq n^{2-o(1)}$. Similar considerations can be used to prove the same lower bound in the cases $e = 7,8$ (see \cite{GS}). All other cases are open; in particular, it is not known whether $f_3(n,9,6) \geq n^{2-o(1)}$ (though see \cite{GS_grids} for results on a related problem). The study of lower bounds naturally leads to the following more general question: for which $e$ and $d$ is it the case that $f_3(n,e+d,e) \leq n^{2-\varepsilon}$, where $\varepsilon > 0$ is a constant depending only on $e,d$? For this question, Gowers and Long~\cite{GL} made the surprising conjecture that $f_3(n,e+4,e) \leq n^{2-\varepsilon}$ for every $e$.  

Finally, let us mention that in the special {\em group setting}, where the hypergraphs under consideration arise from the multiplication table of a group, Conjecture \ref{BESC} has recently been settled in a strong form \cite{Long,NST} (see also \cite{S,SW,Wong}).  

The rest of the paper is organized as follows. Section \ref{sec:proof_overview} is the main hub of the paper.
It contains a (verbal) flowchart of the process which finds the required configurations of edges needed to prove Theorem \ref{thm:main}. This section also contains the (formal) proof of Theorem \ref{thm:main},
assuming two key lemmas that we prove in Sections \ref{sec:first_lemma} and \ref{sec:sarkozyselkow}.
Hence, Sections \ref{sec:proof_overview}, \ref{sec:first_lemma} and \ref{sec:sarkozyselkow} are essentially independent of each other.
Finally, in Section \ref{sec:gen_Ramsey}, we discuss an application of our results to a generalized Ramsey problem of Erd\H{o}s and Gy\'{a}rf\'{a}s which is known to have connections to the Brown--Erd\H{o}s--S\'{o}s problem. 
Throughout the paper, we make no effort to optimize any of the constants involved.
All logarithms are natural unless explicitly stated otherwise.



		
\section{Proof Overview and Proof of Theorem \ref{thm:main}}\label{sec:proof_overview}

Our goal in this section is fourfold. We first give an overview of the proof of Theorem \ref{thm:main}.
In doing so, we will state the two key lemmas, Lemmas \ref{lem:main} and \ref{lem:Sarkozy_Selkow}, used in its proof. We will then proceed to show how these
two lemmas can be used in order to prove Theorem \ref{thm:main}. 
Finally, in Section \ref{subsec:reduction}, we prove Proposition \ref{prop:reduction}.

\subsection{Proof overview and the key lemmas}\label{subsec:overview}

%

Our first simple (yet crucial) observation towards the proof of Theorem \ref{thm:main} is that, in order to prove the theorem, it is enough
to prove the following approximate version.

\begin{lemma}\label{lem:approx}
For every $e \geq 576 = (4!)^2$ and $\varepsilon \in (0,1)$, there is $n_0 = n_0(e,\varepsilon)$ such that every $3$-graph $H$ with $n \geq n_0$ vertices and at least $\varepsilon n^2$ edges contains a $(v',e')$-configuration satisfying $e - \sqrt{e} \leq e' \leq e$
and $v' - e' \leq 12\log e/\log\log e$.
\end{lemma}

In Section \ref{subsec:main} we will show how to quickly derive Theorem \ref{thm:main} from the above lemma.
So let us proceed with the overview of the proof of Lemma \ref{lem:approx}. We will heavily rely on the hypergraph removal lemma, which states the following.

\begin{theorem}[Hypergraph removal lemma \cite{Gowers,NRS,Rodl_Skokan_1,Rodl_Skokan_2}]\label{thm:removal}
For every $k \geq 2$ and $\varepsilon > 0$ there is $\gamma = \gamma(k,\varepsilon) > 0$ such that the following holds. Let $n \geq 1$ and let $J$ be a $k$-uniform $n$-vertex hypergraph which contains a collection of at least $\varepsilon n^k$ pairwise edge-disjoint $(k+1)$-cliques. Then $J$ contains at least $\gamma n^{k+1}$ $(k+1)$-cliques.
\end{theorem}

Let us start by describing the approach of S\'{a}rk\"{o}zy and Selkow \cite{Sarkozy_Selkow}, which roughly proceeds as follows:
suppose one has already proved that every sufficiently large $n$-vertex $3$-graph with $\Omega(n^2)$ edges contains an $(e+k,e)$-configuration. Using this fact, one then
shows that every such $3$-graph also contains a $(2e+k+2,2e+1)$-configuration. In other words, at the price of increasing $v-e$
by $1$, we multiply the number of edges by roughly $2$ (and hence the term $\log_2 e$ in \eqref{eq:Sarkozy_Selkow}). The proof of \cite{Sarkozy_Selkow} used the graph removal lemma (at least implicitly\footnote{We will extend the approach of \cite{Sarkozy_Selkow} in Lemma \ref{lem:Sarkozy_Selkow} by using the (hypergraph) removal lemma explicitly.}).
As we mentioned before, Solymosi and Solymosi \cite{Solymosi}
improved the bound of \cite{Sarkozy_Selkow} for the special case $e=10$. The way they achieved this was by cleverly replacing the application of the graph
removal lemma with an application of the $3$-graph removal lemma. Roughly speaking, this allowed them to multiply a $(6,3)$-configuration by $3$, instead of by $2$ as in \cite{Sarkozy_Selkow}.

The above discussion naturally leads one to try and extend the approach of \cite{Solymosi} by showing that
after multiplying the initial configuration by $3$, one can use the $4$-graph removal lemma to multiply the resulting configuration by $4$, etc.
Performing $k$ such steps should (roughly) give a $(k!+k,k!)$-configuration or, equivalently, a $(v,e)$-configuration with $v-e=O(\log e/\log\log e)$. There is one big challenge
and two problems with this approach. The challenge is of course how to achieve this repeated multiplication process.\footnote{The special case in \cite{Solymosi} of multiplying a $(6,3)$-configuration by $3$ proceeds by case analysis which is not generalizable.}
As to the problems, the first is that we do not know how to guarantee that one can indeed keep multiplying the size
of the configurations. In other words, it is entirely possible that this process might get ``stuck'' along the way (this scenario is described in Item 1 of Lemma \ref{lem:main}).
More importantly, even if the process succeeds in producing a $(k!+k,k!)$-configuration for every $k$, it is not clear
how to ``interpolate'' so as to prove Theorem \ref{thm:main} for values of $e$ with $(k-1)! < e < k!$. That is, our process only guarantees the existence of suitable configurations for a very sparse set of values of $e$. It it tempting to guess that the resulting $(k!+k,k!)$-configurations are ``degenerate'', in the sense that one can repeatedly remove from them vertices of degree $1$, thus maintaining the difference $v - e$ while reducing the number of edges until the desired number is achieved. 
This is however false.
Having said that, we will return to this issue of degeneracy after the statement of Lemma \ref{lem:Sarkozy_Selkow}.

In what follows, it will be convenient to use the following notation.
\begin{definition}
	For a $3$-graph $F$ and $U \subseteq V(F)$, the {\em difference} of $U$ is defined as $\Delta(U) := |U| - e(U)$. We will write $\Delta(F)$ for $\Delta(V(F))$, i.e., $\Delta(F) := v(F) - e(F)$, and call $\Delta(F)$ the difference of $F$.
\end{definition}

Our first key lemma, Lemma \ref{lem:main} below, comes close to achieving the ``iterated multiplication process''
described above. Given an $n$-vertex $3$-graph $H$ with $\Omega(n^2)$ edges,
the lemma almost resolves the aforementioned challenge by either showing that $H$ contains configurations with difference $k$ and size roughly $k!$ (this is the statement of Item 2)
or getting ``stuck'' in the scenario described in Item 1. The silver lining in Item 1 is that we get an arithmetic progression of values $v$ for which we can construct
$(v,e)$-configurations of small difference. The problem is that the common difference of this arithmetic progression might be much larger than $\sqrt{e}$, so this lemma alone cannot be used in order to prove Lemma \ref{lem:approx}.
Before stating Lemma \ref{lem:main}, we need to introduce the following key definition that appears in its statement. 

\begin{definition}\label{def:main}
Let $F$ be a $3$-graph and put $k := \Delta(F) = v(F) - e(F)$. We call $F$ {\em nice} if there is an independent set $A \subseteq V(F)$ of size $k+1$ such that the following holds for every $U \subseteq V(F)$.
\begin{enumerate}
\item $\Delta(U) \geq |U \cap A| - \mathds{1}_{A \subseteq U}$.
\item If $|U \cap A| \leq k-1$ and $U \setminus A \neq \emptyset$, then $\Delta(U) \geq |U \cap A| + 1$.
\end{enumerate}
\end{definition}

One can think of the set $A$ in Definition \ref{def:main} as being {\em sparsifying}. That is, if a set $U$ contains many elements of $A$, then this forces $U$ to be (somewhat) sparse (i.e., it forces $\Delta(U)$ to be large). 
In particular, Item 1 guarantees that if $A \subseteq U$, then $\Delta(U) \geq k = \Delta(F)$. This implies that if we glue multiple copies of $F$ on a set $U \subseteq V(F)$ with $A \subseteq U$, then the resulting hypergraph $F^*$ will satisfy $\Delta(F^*) \leq k$. This property is crucial for our argument, as it will allow us to ``glue'' an arbitrary number of copies of $F$ without increasing the difference, thus achieving the conclusion of Item 1 of Lemma \ref{lem:main} below. 
The condition in Definition \ref{def:main} is more complicated than simply requiring that $\Delta(U) \geq k$ for sets $U$ containing $A$. The reason for this is (partly) that we want the condition to ``carry over'' from $F$ to the hypergraph $F'$ obtained by ``multiplying'' $F$, so that the process can be continued. The exact condition in Definition \ref{def:main} was devised for this purpose.  

\begin{lemma}\label{lem:main}
There is a sequence $(F_k)_{k \geq 3}$ of $3$-graphs such that $\Delta(F_k) = v(F_k) - e(F_k) = k$, $F_k$ is nice for each $k \geq 4$, $e(F_3) = 3$ and $e(F_k) = \nolinebreak 5k!/12$ for each $k \geq 4$, and the following holds. For every $k \geq 4$, $r \geq 1$ and $\varepsilon \in (0,1)$, there are $\eta = \eta_{\ref{lem:main}}(k,r,\varepsilon) \in (0,1)$ and $n_0 = n_0(k,r,\varepsilon)$ such that every $3$-graph $H$ with $n \geq n_0$ vertices and at least $\varepsilon n^2$ edges satisfies (at least) one of the following:
\begin{enumerate}
\item There are $3 \leq j \leq k-1$ and
			$0 \leq q \leq e(F_j)-1$ such that, for every $1 \leq i \leq r$, the $3$-graph $H$ contains a $(v',e')$-configuration with
			$v' - e' \leq j$ and
			$e' = q + i \cdot (e(F_j) - q)$.
\item $H$ contains at least $\eta n^{k}$ copies of $F_k$.
\end{enumerate}
\end{lemma}

\begin{remark}
	Item 2 in Lemma \ref{lem:main} asserts that $H$ contains $\Omega(n^k)$ copies of $F_k$. The exponent $k$ is best possible, as can be seen by considering the random $3$-graph with edge density $\frac{1}{n}$, which has $O(n^{\Delta(F_k)}) = O(n^k)$ copies of $F_k$ w.h.p. A similar situation 
	occurs in Item 3(b) of Lemma \ref{lem:Sarkozy_Selkow}. 
\end{remark}

The proof of Lemma \ref{lem:main} proceeds by induction on $k$. Namely, assuming
$H$ contains $\Omega(n^k)$ copies of $F_{k}$, we show that either $H$ contains $\Omega(n^{k+1})$ copies of $F_{k+1}$ or Item 1 holds.
This is done as follows. 
Recalling that $F_{k}$ is nice, we fix a set $A \subseteq V(F_k)$ of size $|A| = k+1$ which witnesses this fact (see Definition \ref{def:main}). 
By assumption, there are $\Omega(n^k)$ embeddings of $F_k$ into $H$. 
For two embeddings $\varphi,\varphi' : V(F_k) \rightarrow V(H)$, we consider the set $U(\varphi,\varphi') = \{u \in V(F_k) : \varphi(u) = \varphi'(u)\}$ (i.e., the set of elements on which $\varphi$ and $\varphi'$ agree).
By a straightforward argument (combining an application of the multicolor Ramsey theorem with a simple cleaning procedure), we can show that either there are embeddings
$\varphi_1,\dots,\varphi_r : V(F_k) \rightarrow V(H)$ and a set $U \subseteq V(F_k)$ such that $|U| \geq k$, $|U \cap A| \geq k-1$ and
$U(\varphi_i,\varphi_j) = U$ for all $1 \leq i \neq j \leq r$ or there is a family $\mathcal{F}$ of $\Omega(n^{k})$ embeddings $\varphi : V(F_k) \rightarrow V(H)$ such that, for any two $\varphi,\varphi' \in \mathcal{F}$, we have $|U(\varphi,\varphi') \cap A| \leq k-1$ with equality only if $U(\varphi,\varphi') \subseteq A$.
In the former case, Items 1-2 of Definition \ref{def:main} imply that $\Delta(U) \geq k$, which in turn implies that the union of the copies of $F_k$ corresponding to $\varphi_1,\dots,\varphi_r$ has difference at most $k$. From this it then easily follows that Item 1 in Lemma \ref{lem:main} holds. In the latter case, we define an auxiliary $k$-uniform hypergraph by putting a $k$-uniform $(k+1)$-clique on the set $\varphi(A)$ for each $A \in \mathcal{F}$. The aforementioned property of $\mathcal{F}$ implies that these cliques are pairwise edge-disjoint, which allows us to apply the hypergraph removal lemma (Theorem \ref{thm:removal}) and thus infer that the number of $(k+1)$-cliques in our auxiliary hypergraph is $\Omega(n^{k+1})$. Using again our guarantees regarding $\mathcal{F}$, we can show that most such $(k+1)$-cliques correspond to copies of a particular $3$-graph consisting of $k+1$ copies of $F_k$ which do {\em not} intersect outside of the set $A$. This $3$-graph is then chosen as $F_{k+1}$. One of the challenges in the proof is to then show that $F_{k+1}$ is itself nice, thus allowing the induction to continue. The full details appear in Section \ref{sec:first_lemma}.

We now move to our next key lemma, Lemma \ref{lem:Sarkozy_Selkow} below. Let us say that a $3$-graph is $d$-degenerate
if it is possible to repeatedly remove from it a set of at least $d$ vertices which touches at most $d$ edges.
As we mentioned above, the 3-graphs $F_k$ are not $1$-degenerate, so it is not possible
to take
one of these $3$-graphs and repeatedly remove vertices of degree at most $1$ so as to obtain configurations with any desired number of edges, while not increasing the difference.
One can argue, however, that since Lemma \ref{lem:approx}
only asks for $e'$ to satisfy $e-\sqrt{e} \leq e' \leq e$, it is enough to show that the $3$-graphs $F_k$ are, say, $e(F_k)^{1/3}$-degenerate. Unfortunately, we cannot do even this. Instead, we will overcome the problem by using Lemma \ref{lem:Sarkozy_Selkow}.
This lemma states that if $H$ contains many copies of some nice $3$-graph $G$, then it also contains copies of $3$-graphs $G_0 = G,G_1,G_2,\dots$ which are all $e(G)$-degenerate and whose sizes increase. In fact, as in Lemma \ref{lem:main}, we cannot always guarantee success in finding copies of $G_1,G_2,\ldots,G_{\ell}$ in $H$,
since the process might get stuck in a situation analogous to the one in Lemma \ref{lem:main}. Finally, the price we have to pay for the degeneracy
guaranteed by Item 2 of Lemma \ref{lem:Sarkozy_Selkow} is that the size of the $3$-graphs $G_1,G_2,\ldots,G_{\ell}$ only grows by a factor of roughly $k$ at each step, where $k = \Delta(G)$. Namely, $G_{\ell}$ grows (only) exponentially (in $\ell$), unlike the factorial growth in Lemma \ref{lem:main}. 
Hence,
just like Lemma \ref{lem:main}, Lemma \ref{lem:Sarkozy_Selkow} alone also falls short of proving Lemma \ref{lem:approx}.

\begin{lemma}\label{lem:Sarkozy_Selkow}
Let $G$ be a nice $3$-graph, put $k := \Delta(G) = v(G) - e(G)$ and assume that $k \geq 2$. Then there is a sequence of $3$-graphs $(G_{\ell})_{\ell \geq 0}$ having the following properties.
\begin{enumerate}
			\item $G_0 = G$, $\Delta(G_{\ell}) = v(G_{\ell}) - e(G_{\ell}) = k + \ell$ and
			$e(G_{\ell}) =
			\frac{k^{\ell+1} - 1}{k - 1} \cdot e(G)$.
			\item For every $\ell \geq 0$ and every
			$0 \leq t \leq e(G_{\ell})/e(G)$, the $3$-graph $G_{\ell}$ contains a $(v',e')$-configuration with
			$v' - e' \leq k + \ell$ and
			$e' = t \cdot e(G)$.
			\item For every $\ell \geq 0$, $r \geq 1$ and $\varepsilon \in (0,1)$, there are $\delta = \delta_{\ref{lem:Sarkozy_Selkow}}(\ell,r,\varepsilon) \in (0,1)$ and $n_0 = n_0(\ell,r,\varepsilon)$ such that, for every $3$-graph $H$ on $n \geq n_0$ vertices, if $H$ contains at least $\varepsilon n^{k}$ copies of $G$, then (at least) one of the following conditions is satisfied:
			\begin{enumerate}
				\item There are $0 \leq j \leq \ell-1$ and
				$0 \leq q \leq e(G_{j})-1$ such that, for every $1 \leq i \leq r$, the $3$-graph $H$ contains a $(v',e')$-configuration which contains a copy of $G_j$ and satisfies
				$v' - e' \leq k + j$ and
				$e' = q + i \cdot (e(G_{j}) - q)$.
				\item $H$ contains at least $\delta \cdot n^{k + \ell}$ copies of $G_{\ell}$.
			\end{enumerate}
		\end{enumerate}
	\end{lemma}

Strictly speaking, we cannot apply Lemma \ref{lem:Sarkozy_Selkow} with $G$ being an edge, since an edge is not a nice $3$-graph (indeed, it has difference $k = 2$ but evidently contains no independent set of size $k + 1 = 3$). However, one can check that the proof also works when $G$ is an edge (and, more generally, in any case where $k := \Delta(G) = 2$ and one can choose a (not necessarily independent) $A \subseteq V(G)$ of size $3$ which satisfies Items 1-2 in Definition \ref{def:main}). 
By applying Lemma \ref{lem:Sarkozy_Selkow} with $G$ being an edge, one recovers the construction used by S\'{a}rk\"{o}zy and Selkow \cite{Sarkozy_Selkow} to prove \eqref{eq:Sarkozy_Selkow}.
Generalizing this construction to other graphs $G$ (e.g., for $k \geq 3$) presents a challenge, which we overcome by using some of the ideas from the proof of Lemma \ref{lem:main}.

We now sketch the derivation of Lemma \ref{lem:approx} from Lemmas \ref{lem:main} and \ref{lem:Sarkozy_Selkow} (the formal proof appears
in the next subsection). 
Given $e$, choose $k$ such that $k! \approx \sqrt{e}$; then $k = \Theta(\log e/\log \log e)$.
We first apply Lemma \ref{lem:main} with $k$. If we are at Item 1, then we get an arithmetic progression with common difference at most $e(F_{k-1}) \leq k! \leq \sqrt{e}$ of values $e'$ for which we can
find $(v',e')$-configurations of difference $v' - e' \leq k$, thus completing the proof in this case. Suppose then that we are at Item 2, implying that $H$
contains $\Omega(n^{k})$ copies of $F_k$. Since $F_k$ is nice, we can apply Lemma \ref{lem:Sarkozy_Selkow} with
$G=F_k$. 
Since $k = \Theta(\log e/\log\log e)$, for $\ell = O(k)$ we have $k^{\ell} > e$, so that $e(G_{\ell}) \approx e(F_k) \cdot k^{\ell} > e$ (by Item 1 of Lemma \ref{lem:Sarkozy_Selkow}). 
Suppose first that the application of Lemma \ref{lem:Sarkozy_Selkow} results in Item 3(b). Then we can use Item 2 of Lemma \ref{lem:Sarkozy_Selkow} to find inside $G_{\ell}$ a $(v',e')$-configuration of difference $O(k+\ell)=O(k)$ with $e-\sqrt{e} \leq e-e(G) \leq e' \leq e$, thus completing the proof. Now suppose that we are at Item 3(a) and let $0 \leq j \leq \ell-1$ be as in that item. 
Then we can find a $(v',e')$-configuration $G'$ with difference $O(k+\ell)=O(k)$ and $e-e(G_j) \leq e' \leq e$. With the help of a simple trick
we can also find in $H$ a copy $G^*$ of $G_j$ which is {\em edge-disjoint} from $G'$. As in case 3(b) above, we use Item 2 of Lemma \ref{lem:Sarkozy_Selkow} to find a subconfiguration $G''$ of $G^*$
with difference $O(k+\ell) = O(k)$ and $e - e(G') - e(G) \leq e(G'') \leq e - e(G')$. If we now take $G'''$ to be the union of $G'$ and $G''$, we get that $G'''$ has difference $O(k)$
and $e-\sqrt{e} \leq e-e(G) \leq e(G''') \leq e$. So again we are done.

\subsection{Deriving Lemma \ref{lem:approx} from Lemmas \ref{lem:main} and \ref{lem:Sarkozy_Selkow}}\label{sec:thm_proof}
		The required integer $n_0 = n_0(e,\varepsilon)$ will be chosen implicitly.
		%
		Let $(F_k)_{k \geq 3}$ be the nice $3$-graphs whose existence is guaranteed by Lemma \ref{lem:main}. Recall that $e(F_3) = 3$ and $e(F_k) = 5k!/12$ for each $k \geq 4$.
		Let $k \geq 4$ be such that $k! \leq \sqrt{e} < (k+1)!$, and note that $e(F_k) \leq k! \leq \sqrt{e}$.
		We now check\footnote{Asymptotically, $k = (0.5 + o(1))\log e/\log\log e$, but in order to avoid asymptotic expressions in the statements of Lemma \ref{lem:approx} and Theorem \ref{thm:main}, we instead use an (easy to obtain) estimate.} that $k \leq 2\log e/\log\log e$. 
		Set 
		$x = \lceil 2\log e/\log\log e \rceil$. We need to check that $x! > \sqrt{e}$, which implies that $k < x$. 
		It is well-known that $x! \geq (x/\exp)^x$ for every $x \geq 1$, where $\exp$ is the base of the natural logarithm. 
		Hence, it suffices that $(x/\exp)^x > \sqrt{e}$, which is equivalent to $x(\log x - 1) > \frac{1}{2}\log e$. Since $x(\log x - 1)$ is an increasing function for $x \geq 1$, it is enough to plug in $x=2\log e/\log\log e$ and check that the above inequality holds. This inequality becomes $\frac{2\log e}{\log \log e}\left( \log\log e + \log 2 - \log\log\log e - 1 \right) > \frac{1}{2}\log e$. Hence, it suffices that 
		$\log\log e + \log 2 - \log\log\log e - 1 > \frac{1}{4}\log\log e$, which can be shown to hold for all $e \geq 3$.  
		Thus, $k \leq 2\log e/\log\log e$, as required. 

		We will now apply our second construction, given by Lemma \ref{lem:Sarkozy_Selkow}. Set
		$G := F_k$ and let $(G_{\ell})_{\ell \geq 0}$ be the sequence of $3$-graphs whose existence is guaranteed by Lemma \ref{lem:Sarkozy_Selkow}.
		Let $\ell$ be the minimal integer satisfying $e(G_{\ell}) \geq e$.
		Then $\ell \geq 1$ (because $e(G_0) = e(G) = e(F_k) < e$). 
		We now show that $\ell \leq 2k$, that is, $e(G_{2k}) \geq e$. 
		Using Item 1 of Lemma \ref{lem:Sarkozy_Selkow}, we have $e(G_{2k}) = \frac{k^{2k + 1} - 1}{k - 1} \cdot e(G) \geq k^{2k} \geq ((k+1)!)^2 > e$, where the last inequality uses our choice of $k$ and the penultimate inequality holds for every $k \geq 3$. So indeed $\ell \leq 2k$. 
		
Let $H$ be a $3$-graph with $n \geq n_0$ vertices and at least $\varepsilon n^2$ edges.		
Partition $E(H)$ into equal-sized parts $E_1,\dots,E_{\ell+1}$ and, for each $1 \leq i \leq \ell+1$, let $H_i$ be the hypergraph $(V(H),E_i)$. Note that $e(H_i) \geq e(H)/(\ell+1) \geq \varepsilon n^2/(\ell+1)$ for each $1 \leq i \leq \ell+1$.
\begin{claim}\label{claim:approx}
For each $1 \leq m \leq \ell+1$, either $H_m$ satisfies the assertion of Lemma \ref{lem:approx} or there exists $0 \leq j \leq \ell-1$ such that $H_m$ contains a  $(v',e')$-configuration which contains a copy of $G_j$ and satisfies
$v' - e' \leq k + j$ and
$e - e(G_j) \leq e' \leq e$.
\end{claim}
\begin{proof}
For convenience, let us assume that $m = 1$.
We apply Lemma \ref{lem:main} to $H_1$ with parameters $k$, $r = e$ and $\varepsilon/(\ell+1)$. Suppose first that the assertion of Item 1 in Lemma \ref{lem:main} holds and let $3 \leq j \leq k-1$ and $0 \leq q \leq e(F_j) - 1$ be as in that item.
Let $i$ be the maximal integer satisfying $q + i \cdot (e(F_j) - q) \leq e$ and note that
$1 \leq i \leq e$. We may thus infer from Item 1 in Lemma \ref{lem:main} that $H_1$ contains a $(v',e')$-configuration with $e' = q + i \cdot (e(F_j) - q) \leq e$
and
\begin{equation*}\label{eq:v1}
v' - e' \leq j < k \leq 2\log e/\log\log e.
\end{equation*}
Observe that, by the maximality of $i$, we have
$e - e' < e(F_j) - q \leq e(F_j) \leq j! \leq k! \leq \sqrt{e}$. Hence, $H_1$ satisfies the assertion of Lemma \ref{lem:approx}, as required. This completes the proof for the case that Item 1 in Lemma \ref{lem:main} holds.

					
		Suppose from now on that the assertion of Item 2 in Lemma \ref{lem:main} holds, namely, that $H_1$ contains at least $\eta n^{k}$ copies of $F_k = G$. This means that we may apply Lemma \ref{lem:Sarkozy_Selkow} to $H_1$ (with $G = F_k$).
		By Item 3 of Lemma \ref{lem:Sarkozy_Selkow}, applied with $r = e$ and with $\eta$ in place of $\varepsilon$, the $3$-graph $H_1$ satisfies (at least) one of the following:
		\begin{enumerate}
			\item[(a)] There are some
			$0 \leq j \leq \ell-1$ and
			$0 \leq q \leq e(G_{j})-1$ such that, for every
			$1 \leq i \leq e$, $H_1$ contains a  $(v',e')$-configuration which contains a copy of $G_j$ and satisfies
			$v' - e' \leq k + j$ and
			$e' = q + i \cdot (e(G_{j}) - q)$.
			\item[(b)] $H_1$ contains a copy of $G_{\ell}$ (in fact, at least $\delta_{\ref{lem:Sarkozy_Selkow}}(\ell,r,\eta) \cdot n^{k+\ell}$ such copies).
		\end{enumerate}
		
		Suppose first that $H_1$ satisfies Item (b). Let $t \geq 0$ be the maximal integer satisfying $t \cdot e(G) \leq e$ and note that $t \leq e/e(G) \leq e(G_{\ell})/e(G)$, where the second inequality uses our choice of $\ell$.
		By Item 2 of Lemma \ref{lem:Sarkozy_Selkow}, $H_1$ contains a
		$(v',e')$-configuration with
		$v' - e' \leq k + \ell \leq 3k \leq 6\log e/\log\log e$ and
		$e' = t \cdot e(G) \leq e$. By our choice of $t$, we have
		$e - e' < e(G) = e(F_k) = 5k!/12 \leq k! \leq \sqrt{e}$. So in this case the assertion of Lemma \ref{lem:approx} indeed holds for $H_1$.
		
Finally, assume that $H_1$ satisfies Item (a) and
let $0 \leq j \leq \ell - 1$ and $0 \leq q \leq e(G_{j}) - 1$ be as in that item. Let $i$ be the maximal integer satisfying
$q + i \cdot (e(G_{j}) - q) \leq e$. Then $1 \leq i \leq e$.
We may thus rely on (a) above to conclude that $H_1$ contains a $(v',e')$-configuration which contains a copy of $G_j$ and satisfies $e' = q + i \cdot (e(G_{j}) - q) \leq e$ 
and $v' - e' \leq k+j$. 
Observe that the maximality of $i$ guarantees that $e - e' < e(G_j) - q \leq e(G_j)$. 
We conclude that $H_1$ indeed contains a $(v',e')$-configuration with the properties stated in the claim.
\end{proof}

	We now return to the proof of the lemma. If some $H_m$ ($1 \leq m \leq \ell+1$) satisfies the assertion of Lemma \ref{lem:approx}, then we are done. Otherwise, Claim \ref{claim:approx} implies that for each $1 \leq m \leq \ell+1$ there is $0 \leq j_m \leq \ell - 1$ such that $H_m$ contains a $(v',e')$-configuration which contains a copy of $G_{j_m}$ and satisfies
	$v' - e' \leq k + j_m$ and $e - e(G_{j_m}) \leq e' \leq e$. By the pigeonhole principle, there are two indices $1 \leq m \leq \ell+1$ whose $j_m$'s are equal. It follows that, for some $0 \leq j \leq \ell-1$, $H$ contains {\em edge-disjoint} subgraphs $G^*$ and $G'$ such that $G^*$ is isomorphic to $G_j$ and $G'$ satisfies $v(G') - e(G') \leq k + j$ and $e - e(G_j) \leq e(G') \leq e$.
		Let $t$ be the maximal integer satisfying
		$t \cdot e(G) \leq e - e(G')$ and note that $0 \leq t \leq e(G_j)/e(G)$ because $e - e(G') \leq e(G_j)$.
		Then, by Item 2 of Lemma \ref{lem:Sarkozy_Selkow} (applied to the copy $G^*$ of $G_j$ with $j$ in place of $\ell$), there is a subgraph $G''$ of $G^*$ such that $v(G'') - e(G'') \leq k + j$ and $e(G'') = t \cdot e(G)$. Our choice of $t$ implies that
		$0 \leq e - e(G') - e(G'') < e(G) = e(F_k) \leq k! \leq \sqrt{e}$. Let $G'''$ be the union of $G'$ and $G''$, so that $e(G''') = e(G') + e(G'')$. We see that
		$e - \sqrt{e} \leq e(G''') \leq e$ and
		$$
		v(G''') - e(G''') \leq v(G') - e(G') + v(G'') - e(G'') \leq
		2(k + j) \leq
		2(k + \ell) \leq 6k \leq 12 \log e / \log \log e.
		$$
		So we see that the assertion of the lemma holds with $G'''$ as the required $(v',e')$-configuration.

\subsection{Deriving Theorem \ref{thm:main} from Lemma \ref{lem:approx}}\label{subsec:main}
	
		Our goal is to show that for every $e \geq 3$ and $\varepsilon \in (0,1)$, there is $n_0 = n_0(e,\varepsilon)$ such that every $3$-graph with $n \geq n_0$ vertices and at least $\varepsilon n^2$ edges contains a $(v,e)$-configuration with $v - e \leq 26\log e/\log \log e$.
		As in the proof of Lemma \ref{lem:approx}, the required integer $n_0 = n_0(e,\varepsilon)$ will be chosen implicitly.
		The proof is by induction on $e$.
		Let $H$ be a $3$-graph with $n \geq n_0$ vertices and at least $\varepsilon n^2$ edges.
		First, we take care of the case where $e$ is small by using \eqref{eq:Sarkozy_Selkow}.
		Indeed, by \eqref{eq:Sarkozy_Selkow}, $H$ contains a $(v,e)$-configuration with
		$v - e \leq 2 + \lfloor \log_2 e \rfloor$. 
		If $e \leq \exp(2^{24})$, then we have
		$2 + \lfloor \log_2 e \rfloor \leq 2 + 24\log e/\log\log e \leq 26\log e/\log\log e$ (where the second inequality holds whenever $e \geq 3$), thus completing the proof in this case.
		So suppose from now on that $e > \exp(2^{24})$. 
		
		By Lemma \ref{lem:approx}, $H$ contains a $(v',e')$-configuration $F'$ satisfying
		$e - \sqrt{e} \leq e' \leq e$ and $v' - e' \leq 12\log e/\log\log e$.
		Set $e'' := e - e'$, noting that $0 \leq e'' \leq \sqrt{e}$. If $e'' \leq 15$, then, by adding at most $15$ edges to $F'$, one obtains a $(v,e)$-configuration with
		$v - e \leq v' + 3e'' - (e' + e'') = v' - e' + 2e'' \leq 12\log e/\log\log e + 30 \leq 26\log e/\log\log e$, as required. (Here the last inequality holds for every $e \geq 3$.) 
		So suppose from now on that $e'' \geq 16$.
		Let $H'$ be the $3$-graph obtained from $H$ by deleting the edges of $F'$. Since
		$e(H') \geq e(H) - e(F') \geq \varepsilon n^2 - e(F') \geq \frac{\varepsilon}{2}n^2$ (provided that $n$ is large enough), we may apply the induction hypothesis to $H'$, with parameter $e''$ in place of $e$, and thus obtain a $(v'',e'')$-configuration $F''$ which is edge-disjoint from $F'$ (because it is contained in $H'$) and satisfies
		$$
		v'' - e'' \leq \frac{26 \log e''}{\log \log e''} \leq
		\frac{26 \log \sqrt{e}}{\log \log \sqrt{e}} =
		\frac{13\log e}{\log\log e - \log 2} \leq \frac{14\log e}{\log\log e}.
		$$
		Here, the second inequality uses the fact that the function
		$x \mapsto \log x/\log\log x$ is monotone increasing for $x \geq 16$ and the last inequality holds whenever $e \geq \exp(2^{14})$.
	    Letting $F$ be the union of $F'$ and $F''$, we see that $e(F)=e(F')+e(F'')= e$ and $v(F) \leq v(F') + v(F'')$, implying that
		$$
		v(F) - e(F) \leq v(F') - e(F') + v(F'') - e(F'') \leq
		\frac{12\log e}{\log\log e} + \frac{14\log e}{\log\log e} =
		\frac{26\log e}{\log \log e}.
		$$
		This completes the proof of the theorem.
	
	\subsection{Proof of Proposition \ref{prop:reduction}}\label{subsec:reduction}
	Let $2 \leq k < r$, $e \geq 3$ and $d \geq 1$.
	Let $H$ be an $n$-vertex $r$-graph with
	$$e(H) > \left( \binom{r-k+2}{3} \cdot (e-1) + 1 \right) \cdot \frac{\binom{n}{k-2}}{\binom{r}{k-2}} \cdot f_3(n,e + 2 + d,e).$$
	Our goal is to show that $H$ contains a $(v,e)$-configuration with $v \leq (r-k)e + k + d$.
	By averaging, there are vertices $v_1,\dots,v_{k-2}$ such that more than $\left( \binom{r-k+2}{3} \cdot (e-1) + 1 \right) \cdot f_3(n,e + 2 + d,e)$ of the edges of $H$ contain $v_1,\dots,v_{k-2}$.
	Set
	$E_0 = \{X \setminus \{v_1,\dots,v_{k-2}\} : v_1,\dots,v_{k-2} \in X \text{ and } X \in E(H)\},$
	noting that $|E_0| > \left( \binom{r-k+2}{3} \cdot (e-1) + 1 \right) \cdot f_3(n,e + 2 + d,e)$ and that $|Y| = r - k + 2$ for each $Y \in E_0$.
	
	We now consider two cases. Suppose first that there is a triple
	$T \in \binom{V(H)}{3}$ and distinct $Y_1,\dots,Y_e \in E_0$ such that
	$T \subseteq Y_i$ for each $1 \leq i \leq e$. Setting
	$X_i = Y_i \cup \{v_1,\dots,v_{k-2}\}$ for each $1 \leq i \leq e$, we observe that
	$|X_1 \cup \dots \cup X_e| \leq (r - k - 1) \cdot e + k - 2 + 3 \leq (r-k)e + k$. It follows that $H$ contains a $(v,e)$-configuration with $v \leq (r-k)e + k$, thus completing the proof in this case.
	
	Suppose now that for each $T \in \binom{V(H)}{3}$ it holds that
	$\#\{Y \in E_0 : T \subseteq Y\} \leq e - 1$. Then, for each $Y \in E_0$, there are at most $\binom{r-k+2}{3}(e-1)$ sets $Y' \in E_0 \setminus \{Y\}$ such that
	$|Y \cap Y'| \geq 3$. This means that there exists $E_1 \subseteq E_0$ of size
	\begin{equation}\label{eq:reduction_projection}
	|E_1| \geq \frac{|E_0|}{\binom{r-k+2}{3}(e-1) + 1} > f_3(n,e+2+d,e)
	\end{equation}
	such that $|Y \cap Y'| \leq 2$ for each pair of distinct $Y,Y' \in E_1$.\footnote{	To see that such an $E_1$ indeed exists, consider an auxiliary graph on $E_0$ in which $Y,Y'$ are adjacent if and only if $|Y \cap Y'| \geq 3$ and recall the simple fact that every graph $G$ contains an independent set of size at least $\frac{v(G)}{\Delta(G)+1}$ (where $\Delta(G)$ is the maximum degree of $G$). Now take $E_1$ to be such an independent set.}	
	For each
	$Y \in E_1$, choose arbitrarily a triple $T_Y \in \binom{Y}{3}$. Let $H'$ be the $3$-graph on $V(H)$ whose edge-set is $E(H') = \{T_Y : Y \in E_1\}$. Then $e(H') = |E_1| > f_3(n,e+2+d,e)$, where the equality holds due to our choice of $E_1$ and the inequality due to \eqref{eq:reduction_projection}. It follows that $H'$ contains an $(e+2+d,e)$-configuration $F$. Now observe that the edge-set
	$\{Y \cup \{v_1,\dots,v_{k-2}\} : Y \in E_1 \text{ and } T_{Y} \in E(F) \}$
	spans in $H$ a $(v,e)$-configuration with
	$
	v \leq v(F) + (r - k - 1)e + k - 2 \leq
	e + 2 + d + (r - k - 1)e + k - 2 = (r - k)e + k + d,
	$
	as required.
	
\section{Proof of Lemma \ref{lem:main}}\label{sec:first_lemma}
	In this section we prove Lemma \ref{lem:main}. The construction of the $3$-graphs $F_k$ appearing in the statement of the lemma, as well as the proof that these $3$-graphs have the required properties, is done by induction on $k$.
	The inductive step, which constitutes the main part of the proof of Lemma \ref{lem:main}, is given by the following lemma.
	\begin{lemma}\label{lem:super_exponential}
		Let $F$ be a nice $3$-graph, put $k = v(F) - e(F)$ and assume $k \geq 3$. Then there exists a nice $3$-graph $F'$ such that $v(F') - e(F') = k+1$, 
		$e(F') = (k+1) \cdot e(F)$ and the following holds.
		For every $r \geq 1$ and $\varepsilon \in (0,1)$, there are $\delta = \delta_{\ref{lem:super_exponential}}(F,r,\varepsilon) \in (0,1)$ and $n_0 = n_0(F,r,\varepsilon)$ such that every $3$-graph $H$ with $n \geq n_0$ vertices and at least $\varepsilon n^{k}$ copies of $F$ satisfies (at least) one of the following:
		\begin{enumerate}
			\item There is $0 \leq q \leq e(F) - 1$ such that, for every $1 \leq i \leq r$, $H$ contains a $(v',e')$-configuration with $v'-e' \leq k$ and
			$e' = q + i\cdot(e(F) - q)$.
			\item $H$ contains at least $\delta n^{k+1}$ copies of $F'$.
		\end{enumerate}
	\end{lemma}

		\begin{figure}\label{Fig:(6,3)}
		\centering
		\begin{tikzpicture}[scale = 2.5]
		\coordinate (v1) at (0,1.73);
		\coordinate (v2) at (0.5,0.865);
		\coordinate (v3) at (-0.5,0.865);
		\coordinate (v4) at (0,0);
		\coordinate (v5) at (1,0);
		\coordinate (v6) at (-1,0);
		
		\draw (v1) node[fill=black,circle,minimum size=2pt,inner sep=0pt,label=below:$v_1$] {};
		
		\foreach \i in {2,3,4,5,6}
		{
			\draw (v\i) node[fill=black,circle,minimum size=2pt,inner sep=0pt,label=$v_\i$] {};
		}
		
		
		\draw [rounded corners = 18] (0.8,-0.2) -- (1.35,0) -- (0.35,1.95) -- (-0.2,1.75) -- cycle;
		
		\draw [rounded corners = 18] (-0.8,-0.2) -- (-1.35,0) -- (-0.35,1.95) -- (0.2,1.75) -- cycle;
		
		\draw [rounded corners = 16] (-1.2,0.05) -- (-1.2,-0.4) -- (1.2,-0.4) -- (1.2,0.05) -- cycle;
		\end{tikzpicture}
		\caption{The $3$-uniform linear $3$-cycle}
	\end{figure}
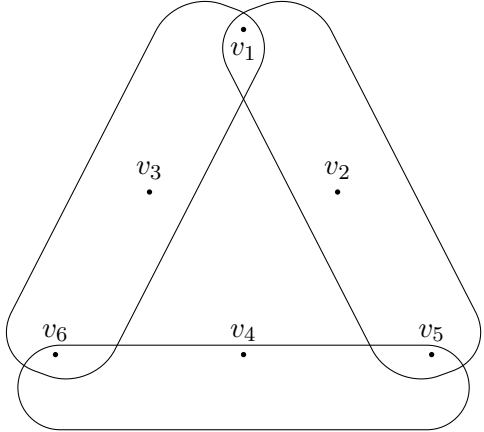
	
	Ideally, we would like to start the induction by invoking Lemma \ref{lem:super_exponential} with $F$ being an edge (so $k = \Delta(F) = 2$). As is the case with Lemma \ref{lem:Sarkozy_Selkow} (see the remark following that lemma), Lemma \ref{lem:super_exponential} does in fact work with $F$ being an edge, even though an edge is not nice as per Definition \ref{def:main}. The $3$-graph $F'$ supplied by Lemma \ref{lem:super_exponential} (when applied with $F$ being an edge) is the linear $3$-cycle (see Figure 1). In fact, applying Lemma \ref{lem:super_exponential} with $F$ being an edge recovers the proof of the (6,3)-theorem. 
	Unfortunately, the linear $3$-cycle is not nice (this time in a meaningful way; it really cannot be used as an input to Lemma \ref{lem:super_exponential}), 
	preventing us from continuing the induction. To make matters even worse, we know of no $3$-graph $F$ with difference $k = 3$ which can be used as a viable input to Lemma \ref{lem:super_exponential}. Indeed, note that in order for the lemma to be useful when applied with input $F$, we need to know that $F$ is abundant\footnote{We say that a $3$-graph $F$ is {\em abundant} in an $n$-vertex $3$-graph $H$ if $H$ contains $\Omega(n^{v(F) - e(F)})$ copies of $F$. In particular, a single edge is trivially abundant in every hypergraph with $\Omega(n^2)$ edges and the condition (resp. conclusion) of Lemma \ref{lem:super_exponential} can be stated as saying that $F$ (resp. $F'$) is abundant in $H$.} in every sufficiently large $n$-vertex $3$-graph with $\Omega(n^2)$ edges (or at least in every such $3$-graph that does not satisfy the conclusion of Theorem \ref{thm:main} for some other reason). Unfortunately, no such nice $F$ (of difference $3$) is known.

	In light of this situation, the base step of our induction will have to involve a nice $3$-graph $F$ with difference at least $4$. Fortunately, as stated in the following lemma, there does exist a nice $F$ of difference $4$ which can be shown to be abundant in every $3$-graph $H$ with $n$ vertices and $\Omega(n^2)$ edges, unless $H$ satisfies the assertion of Theorem \ref{thm:main} for a trivial reason.
	
	\begin{lemma}\label{lem:base_case}
		There is a nice $3$-graph $F$ with $v(F) = 14$ and $e(F) = 10$ having the following property.
		For every $r \geq 1$ and $\varepsilon \in (0,1)$, there are $\delta = \delta_{\ref{lem:base_case}}(r,\varepsilon) \in (0,1)$ and $n_0 = n_0(r,\varepsilon)$ such that every $3$-graph $H$ with $n \geq n_0$ vertices and at least $\varepsilon n^2$ edges satisfies (at least) one of the following:
		\begin{enumerate}
			\item For every $1 \leq i \leq r$, $H$ contains a $(3i+3, 3i)$-configuration.
			\item $H$ contains at least $\delta n^4$ copies of $F$.
		\end{enumerate}
	\end{lemma}
	
	We note that the $3$-graph $F$ in the above lemma played a key role in the proof of Solymosi and Solymosi~\cite{Solymosi} that $f_3(n,14,10) = o(n^2)$. As such, the abundance statement regarding $F$ was already proved in \cite{Solymosi}. Consequently, our main task in the proof of Lemma \ref{lem:base_case} is to show that $F$ is nice.
	
	The rest of this section is organized as follows. In Section \ref{subsec:first_lemma}, we derive Lemma \ref{lem:main} from Lemmas \ref{lem:super_exponential} and \ref{lem:base_case}. We then prove these two lemmas in Sections \ref{subsec:induction_step} and \ref{subsec:base}, respectively.
	
	\subsection{Deriving Lemma \ref{lem:main} from Lemmas \ref{lem:super_exponential} and \ref{lem:base_case}}\label{subsec:first_lemma}
		Let $F_3$ be the linear $3$-cycle (which has $6$ vertices and $3$ edges).
		Let $F_4$ be the nice $3$-graph whose existence is guaranteed by Lemma \ref{lem:base_case}. For each $k \geq 5$, let $F_k$ be the nice $3$-graph $F'$ obtained by applying Lemma \ref{lem:super_exponential} with $F := F_{k-1}$.
		Then it is easy to check by induction that, for every $k \geq 4$, it holds that $v(F_k) - e(F_k) = k$ and $e(F_k) = 5k!/12$. Moreover, $F_k$ is nice for every $k \geq 4$.
		
		Let $r \geq 1$ and $\varepsilon \in (0,1)$. We define a sequence $(\delta_k)_{k \geq 4}$ as follows. Let $\delta_4 = \delta_{\ref{lem:base_case}}(r,\varepsilon)$ be defined via Lemma \ref{lem:base_case} and, for each $k \geq 4$, let
		$\delta_{k+1} = \delta_{\ref{lem:super_exponential}}\left( F_k,r,\delta_k \right)$ be defined via Lemma \ref{lem:super_exponential}. We now show by induction on $k \geq 4$ that the assertion of Lemma \ref{lem:main} holds with $\eta = \eta_{\ref{lem:main}}(k,r,\varepsilon) := \delta_k$.
		For $k = 4$, Lemma \ref{lem:base_case} readily implies that $H$ either satisfies the assertion of Item 2 of Lemma \ref{lem:main} or satisfies the assertion of Item 1 with $j = 3$ and $q = 0$.
		Let now $k \geq 4$, assume that the assertion of Lemma \ref{lem:main} holds for $k$, and let us show that it also holds for $k+1$. By the assertion for $k$, $H$ satisfies the assertion of (at least) one of the items 1-2 of Lemma \ref{lem:main}. If Item 1 is satisfied, then it is also satisfied when $k$ is replaced with $k+1$, so we are done. Suppose then that $H$ satisfies the assertion of Item 2, namely, that $H$ contains at least $\delta_k \cdot n^{k}$ copies of $F_k$. Then, by Lemma \ref{lem:super_exponential} (with parameters $F = F_k$ and $\delta_k$ in place of $\varepsilon$), either $H$ satisfies the assertion of Item 1 in Lemma \ref{lem:main} for parameter $k+1$ (with $j = k$) or it contains at least $\delta_{k+1} \cdot n^{k+1} = \eta_{\ref{lem:main}}(k+1,r,\varepsilon) \cdot n^{k+1}$ copies of $F_{k+1}$, as required by Item 2.
	
	\subsection{Proof of Lemma \ref{lem:super_exponential}}\label{subsec:induction_step}
		Let $A \subseteq V(F)$ be as in Definition \ref{def:main}. It will be convenient to set $v := v(F)$ and to assume (without loss of generality) that $V(F) = [v]$ and $A = [k+1] \subseteq [v]$.
		The required nice $3$-graph $F'$ is defined as follows: take vertices $x_1,\dots,x_{k+1},x'_1,\dots,x'_{k+1}$ and, for each $1 \leq i \leq k+1$, add a copy $F_i$ of $F$ in which $x_j$ plays the role of $j \in V(F)$ for each $j \in [k+1] \setminus \{i\}$, $x'_i$ plays the role of $i \in V(F)$ and all other $v(F) - k - 1$ vertices are new.
		
		Let us calculate the number of vertices and edges in $F'$. First, as
		$A \subseteq V(F)$ is independent, the copies $F_1,\dots,F_{k+1}$ (which comprise $F'$) do not share edges. Hence, $e(F') = (k+1) \cdot e(F)$. Second, we have
		$
		v(F') = k + 1 + (k+1) \cdot (v(F) - k) = k + 1 + (k+1) \cdot e(F) = e(F') + k + 1,
		$
		as required.
		
		We now show that $F'$ is nice. We will show that $F'$ satisfies the requirements of Definition \ref{def:main} with respect to the set $A' := \{x'_1,\dots,x'_{k+1},x_1\}$. (We remark that in the definition of $A'$ we could replace $x_1$ with any other vertex among $x_1,\dots,x_{k+1}$.) For the rest of the proof, we set $X = \{x_1,\dots,x_{k+1}\}$, $X' = \{x'_1,\dots,x'_{k+1}\}$ and
		$A_i = (X \setminus \{x_i\}) \cup \{x'_i\}$ for each $1 \leq i \leq k+1$.
		Observe that, for each $1 \leq i \leq k + 1$, the vertices of $A_i$ are precisely the vertices which play the roles of the vertices of $A = \{1,\dots,k+1\} \subseteq V(F)$ in the copy $F_i$ of $F$.

		It is evident that $|A'| = k+2$
		and easy to see that $A'$ is independent in $F'$.
		Our goal then is to show that every $U \subseteq V(F')$ satisfies the assertion of Items 1-2 in Definition \ref{def:main} (with $A'$ in place of $A$).
		So let $U \subseteq V(F')$ and put $U_i = U \cap V(F_i)$ for each
		$1 \leq i \leq k+1$. Since every vertex of $X$ belongs to exactly $k$ of the copies $F_1,\dots,F_{k+1}$ and every other vertex of $F'$ belongs to exactly one of these copies, we have
		$$
		|U| = \sum_{i = 1}^{k+1}{|U_i|} - (k-1)|U \cap X|.
		$$
		Since $F_1,\dots,F_{k+1}$ are pairwise edge-disjoint, we have
		$$
		e(U) = \sum_{i = 1}^{k+1}{e(U_i)}.
		$$
		Subtracting the above two equalities, we get
		\begin{equation}\label{eq:Delta(U)}
		\Delta(U) = \sum_{i = 1}^{k+1}{\Delta(U_i)} - (k-1)|U \cap X|.
		\end{equation}
		For each $1 \leq i \leq k+1$, it follows from the niceness of $F$ (and the fact that $A_i$ plays the role of $A$ in the copy $F_i$ of $F$) that
		\begin{equation}\label{eq:Delta(U_i)}
		\Delta(U_i) \geq |U_i \cap A_i| - \mathds{1}_{A_i \subseteq U_i}.
		\end{equation}
		Setting $s := \#\{1 \leq i \leq k+1 : A_i \subseteq U_i\}$, we plug \eqref{eq:Delta(U_i)} into \eqref{eq:Delta(U)} to obtain
		\begin{equation}\label{eq:Delta(U)_1}
		\begin{split}
		\Delta(U) &\geq
		\sum_{i = 1}^{k+1}{|U_i \cap A_i|} - (k-1)|U \cap X| - s
		=
		|U \cap X| + |U \cap X'| - s
		\\ &=
		|U \cap A'| + |U \cap \{x_2,\dots,x_{k+1}\}| - s.
		\end{split}
		\end{equation}
		The first equality in \eqref{eq:Delta(U)_1} holds because $A_1 \cup \dots \cup A_{k+1} = X \cup X'$ and every element of $X$ (resp. $X'$) belongs to exactly $k$ (resp. 1) of the sets $A_1,\dots,A_{k+1}$. 
		The second equality in \eqref{eq:Delta(U)_1} is immediate from the definition of $A'$. 
		
		We first prove that
		$\Delta(U) \geq |U \cap A'| - \mathds{1}_{A' \subseteq U}$, as required by Item 1 in Definition \ref{def:main}.
		If $s = 0$, then \eqref{eq:Delta(U)_1} readily gives $\Delta(U) \geq |U \cap A'|$. Suppose then that $s \geq 1$ and let $1 \leq i \leq k+1$ be such that $A_i \subseteq U_i$. Then
		$\{x_1,\dots,x_{k+1}\} \setminus \{x_i\} \subseteq U$, implying that
		$|U \cap \{x_2,\dots,x_{k+1}\}| \geq k-1$. Furthermore, if $s \geq 2$, then $\{x_1,\dots,x_{k+1}\} \subseteq U$, so that
		$|U \cap \{x_2,\dots,x_{k+1}\}| = k$. 
		Thus, by \eqref{eq:Delta(U)_1}, if $s = 1$, then $\Delta(U) \geq |U \cap A'| + k-2$, while if $s \geq 2$, then $\Delta(U) \geq |U \cap A'| + k - s$. 
		This in particular means that if $1 \leq s \leq k-1$, then
		$\Delta(U) \geq |U \cap A'| + 1$, as $k \geq 3$ by the assumption of the lemma. It also follows that $\Delta(U) \geq |U \cap A'| - \mathds{1}_{s = k+1}$. 
	 	Observe that if $s = k+1$, then $A_i \subseteq U_i$ for every $1 \leq i \leq k+1$, so that $X \cup X' \subseteq U$ and hence $A' \subseteq U$. So we indeed have $\Delta(U) \geq |U \cap A'| - \mathds{1}_{A' \subseteq U}$, \nolinebreak as required.
		
		Next, we assume that $|U \cap A'| \leq k$ and $U \setminus A' \neq \emptyset$ and show that in this case $\Delta(U) \geq |U \cap A'| + 1$ (as required by Item 2 in Definition \ref{def:main}). The assumption that $|U \cap A'| \leq k$ implies that $s \leq k-1$, because if $s \geq k$, then $|U \cap X'| \geq k$ and $x_1 \in U$, which means that $|U \cap A'| \geq k+1$. We already saw that $\Delta(U) \geq |U \cap A'| + 1$ if $1 \leq s \leq k-1$,
		so it remains to handle the case that $s = 0$, namely, that $A_i \not\subseteq U_i$ for each $1 \leq i \leq k+1$. If 
		$U \cap \{x_2,\dots,x_{k+1}\} \neq \emptyset$, then \eqref{eq:Delta(U)_1} readily implies that $\Delta(U) \geq |U \cap A'| + 1$ (as $s = 0$). So suppose that $U \cap \{x_2,\dots,x_{k+1}\} = \emptyset$.
		Since $U \setminus A' \neq \emptyset$, there is $1 \leq j \leq k+1$ such that $U_j \setminus A' \neq \emptyset$. 
		Then $U_j \setminus A_j \neq \emptyset$ because $A_j \subseteq A' \cup \{x_2,\dots,x_{k+1}\}$ and $U \cap \{x_2,\dots,x_{k+1}\} = \emptyset$. 
		Moreover, $|U_j \cap A_j| \leq 2 \leq k-1$ because $A_j \setminus \{x_2,\dots,x_{k+1}\} \subseteq \{x_1,x'_j\}$ and $U \cap \{x_2,\dots,x_{k+1}\} = \emptyset$. 
		Now it follows from the niceness of $F$ (or, more precisely, of the copy $F_j$ of $F$) that $\Delta(U_j) \geq |U_j \cap A_j| + 1$, by Item 2 of Definition \ref{def:main}. 
		Moreover, by \eqref{eq:Delta(U_i)}, we have $\Delta(U_i) \geq |U_i \cap A_i|$ for each $1 \leq i \leq k+1$ (here we use that $A_i \not\subseteq U_i$ for each $1 \leq i \leq k+1$).
		By plugging these facts into \eqref{eq:Delta(U)}, in a manner similar to the derivation of \eqref{eq:Delta(U)_1}, we obtain
		$$
		\Delta(U) \geq |U_j \cap A_j| + 1 +
		\sum_{i \in [k+1] \setminus \{j\}}{|U_i \cap A_i|} - (k-1)|U \cap X| =
		|U \cap X| + |U \cap X'| + 1 \geq |U \cap A'| + 1,
		$$
		as required.
		
		Having proved that $F'$ is nice, we go on to show that the assertion of the lemma holds. Given $r \geq 1$ and $\varepsilon \in (0,1)$, we set
		$$
		\delta = \delta_{\ref{lem:super_exponential}}(F,r,\varepsilon) =
		\frac{1}{2}
		\gamma\left( k, \; 2^{-v (1 + 2^{v}r)} \cdot v^{-v} \cdot \varepsilon \right)
		$$
		and $n_0 = n_0(F,r,\varepsilon) = 1/\delta$,
		where $\gamma$ is from Theorem \ref{thm:removal} and $v = v(F)$ as before.
		
		Let $H$ be a $3$-uniform hypergraph with $n \geq n_0$ vertices and
		at least $\varepsilon n^{k}$ copies of $F$.
		Partition the vertices of $H$ randomly into sets $C_1,\dots,C_{v}$ by choosing, for each vertex $x \in V(H)$, a part $C_i$ ($1 \leq i \leq v$) uniformly at random and independently (of the choices made for all the other vertices of $H$) and placing $x$ in this part. A copy of $F$ in $H$ will be called {\em good} if, for each $i = 1,\dots,v$, the vertex playing the role of $i$ in this copy is in $C_i$.
		Since $H$ contains at least $\varepsilon n^k$ copies of $F$, there are in expectation at least $v^{-v} \cdot \varepsilon n^k$ good copies of $F$. So fix a partition $C_1,\dots,C_{v}$ with at least this number of good copies of $F$ and denote the set of these copies by $\mathcal{F}$.
		It will be convenient to identify each good copy of $F$ with the corresponding embedding
		$\varphi : V(F) \rightarrow V(H)$ which maps each $i \in [v] = V(F)$ to a vertex in $C_i$.
		So we will assume that the elements of $\mathcal{F}$ are such mappings.
		
		We now define an auxiliary graph $\mathcal{G}$ on $\mathcal{F}$ as follows: for each pair $\varphi_1,\varphi_2 \in \mathcal{F}$, we let $\{\varphi_1,\varphi_2\}$ be an edge in $\mathcal{G}$ if and only if the set $U := U(\varphi_1,\varphi_2) := \{i \in V(F) : \varphi_1(i) = \varphi_2(i)\}$ satisfies either
		$|U \cap A| \geq k$ or $|U \cap A| = k-1$ and $U \setminus A \neq \emptyset$. Here, $A \subseteq V(F)$ is the set from Definition \ref{def:main}. 
		We distinguish between two cases. Suppose first that there is $\varphi \in \mathcal{F}$ whose degree in $\mathcal{G}$ is at least
		$$
		d := 2^{v (1 + 2^{v}r)}.
		$$
		Let $\varphi_1,\dots,\varphi_d$ be distinct neighbours of $\varphi$ in $\mathcal{G}$.
		By the pigeonhole principle, there is $I_0 \subseteq [d]$ of size at least $2^{-v}d = 2^{v 2^{v}r}$ and a set $U_0 \subseteq V(F)$ such that, for all $i \in I_0$, it holds that $U(\varphi,\varphi_i) = U_0$. Note that, by the definition of $\mathcal{G}$, we have either $|U_0 \cap A| \geq k$ or $|U_0 \cap A| = k-1$ and $U_0 \setminus A \neq \emptyset$.
		We now consider the complete graph on $I_0$ and color each edge $\{i,j\} \in \binom{I_0}{2}$ of this graph with color $U(\varphi_i,\varphi_j)$.
		A well-known bound for multicolor Ramsey numbers (see \cite{CFS}) implies that in every $c$-coloring of the edges of the complete graph on $c^{cr}$ vertices, there is a monochromatic complete subgraph on $r$ vertices. Applying this claim with $c = 2^{v}$, we conclude that there is $I \subseteq I_0$ of size $|I| = r$ and a set $U \subseteq V(F)$ such that $U(\varphi_i,\varphi_j) = U$ for all $\{i,j\} \in \binom{I}{2}$. Observe that, for each $\{i,j\} \in \binom{I}{2}$, we have $U = U(\varphi_i,\varphi_j) \supseteq U(\varphi,\varphi_i) \cap U(\varphi,\varphi_j) = U_0$. This implies that either $|U \cap A| \geq k$ or $|U \cap A| = k-1$ and $U \setminus A \neq \emptyset$. Our choice of $A$ via Definition \ref{def:main} implies that in both cases $\Delta(U) \geq k$. Note also that $U \neq V(F)$ because the copies of $F$ corresponding to $(\varphi_i : i \in I)$ are \nolinebreak distinct.
		
		We now show that the assertion of Item 1 in the lemma holds.
		Suppose without loss of generality that $I = \{1,\dots,r\}$ and write $V_i := \varphi_i(V(F) \setminus U) \subseteq V(H)$ for $1 \leq i \leq r$. Note that $V_1,\dots,V_r$ are pairwise disjoint. We also put $W := \varphi_1(U) = \dots = \varphi_r(U)$.
		Now, fix any $1 \leq i \leq r$ and set $V := V_1 \cup \dots \cup V_i \cup W$. Then
		$|V| = |U| + i \cdot (v(F) - |U|) = i \cdot v(F) - (i-1) \cdot |U|$ and
		$
		e_H(V) \geq e_{F}(U) + i \cdot (e(F) - e_{F}(U)) = i \cdot e(F) - (i-1) \cdot e_{F}(U).
		$
		It follows that
		\begin{align*}
		|V| - e_{H}(V) &\leq i \cdot (v(F) - e(F)) - (i-1)(|U| - e_{F}(U))
		=
		i \cdot k - (i-1) \cdot \Delta(U)
		\\ &\leq i \cdot k - (i-1) \cdot k = k.
		\end{align*}
		Setting $q := e_F(U)$, note that $q \leq e(F) - 1$ because $e_F(U) = |U| - \Delta(U) \leq |U| - k \leq v(F) - 1 - k = e(F) - 1$, as $U \neq V(F)$. From the above, we see that $H[V]$ contains a $(v',e')$-configuration with $v' - e' \leq k$ and $e' = q + i \cdot (e(F) - q)$. 
		So the assertion of Item 1 of the lemma holds with this choice of $q$. 
		This completes the proof in the case that $\mathcal{G}$ has a vertex of degree at least $d$.
		
		From now on we assume that the maximum degree of $\mathcal{G}$ is strictly smaller than $d$ and prove that the assertion of Item 2 in the lemma holds.
		Let $\mathcal{F}^* \subseteq \mathcal{F}$ be an independent set\footnote{Here we use the simple fact (which was already used in Section \ref{subsec:reduction}) that every graph $G$ has an independent set of size at least $v(G)/(\Delta(G) + 1)$, where $\Delta(G)$ is the maximum degree of $G$.} in $\mathcal{G}$ of size at least $v(\mathcal{G})/d = |\mathcal{F}|/d$.
		Recall that we identify $V(F)$ with $[v]$ and $A$ with $[k+1]$.
		We now define an auxiliary $k$-uniform $(k+1)$-partite hypergraph $J$ with parts $C_1,\dots,C_{k+1}$, as follows. For each $\varphi \in \mathcal{F}^*$, put a $k$-uniform $(k+1)$-clique in $J$ on the vertices $\varphi(1) \in C_1,\dots,\varphi(k+1) \in C_{k+1}$. We denote this clique by $K_{\varphi}$.
		Note that, by the definition of $J$, every edge of $J$ is contained in a copy of $F$ in $H$, which corresponds to some embedding $\varphi \in \mathcal{F}^*$.
		
		We claim that the cliques $(K_{\varphi} : \varphi \in \mathcal{F}^*)$ are pairwise edge-disjoint. So fix any distinct $\varphi_1,\varphi_2 \in \mathcal{F}^*$ and suppose, for the sake of contradiction, that the cliques $K_{\varphi_1},K_{\varphi_2}$ share an edge. Then there is $W \subseteq A = [k+1]$ of size $|W| = k$ such that $\varphi_1(i) = \varphi_2(i)$ for every $i \in W$. It follows that $W \subseteq U := U(\varphi_1,\varphi_2)$ and hence $|U \cap A| \geq |W| = k$.
		But this means that $\varphi_1$ and $\varphi_2$ are adjacent in $\mathcal{G}$, contradicting the fact that $\mathcal{F}^*$ is an independent set of $\mathcal{G}$.
		So the cliques $(K_{\varphi} : \varphi \in \mathcal{F}^*)$ are indeed pairwise edge-disjoint. This means that $J$ contains a collection of $|\mathcal{F}^*| \geq |\mathcal{F}|/d
		\geq 2^{-v (1 + 2^{v}r)} \cdot v^{-v} \cdot \varepsilon n^k$
		pairwise edge-disjoint $(k+1)$-cliques. By Theorem \ref{thm:removal} and our choice of $\delta = \delta(F,r,\varepsilon)$, the number of $(k+1)$-cliques in $J$ is at least $2\delta n^{k+1}$.
		
		A $(k+1)$-clique $K$ in $J$ is called {\em colorful} if it is not equal to $K_{\varphi}$ for any $\varphi \in \mathcal{F}^*$. Note that all but at most $n^k$ of the $(k+1)$-cliques in $J$ are colorful (because the non-colorful cliques are pairwise edge-disjoint). It follows that $J$ contains at least $2\delta n^{k+1} - n^k \geq \delta n^{k+1}$ colorful $(k+1)$-cliques (here we use our choice of $n_0$).
		
		We will show that each colorful $(k+1)$-clique gives rise to a copy of $F'$ in $H$. 
		So fix any colorful $(k+1)$-clique $K = \{c_1,\dots,c_{k+1}\}$, where $c_i$ is the unique vertex in $K \cap C_i$ for each $1 \leq i \leq k+1$. By the definition of $J$, for each $i \in [k+1]$, there is $\varphi_i \in \mathcal{F}^*$ such that $\varphi_i(j) = c_j$ for every $j \in [k+1] \setminus \{i\}$.
		We claim that $\varphi_1,\dots,\varphi_{k+1}$ are pairwise distinct. Suppose, for the sake of contradiction, that $\varphi_i = \varphi_{i'} =: \varphi$ for some $1 \leq i < i' \leq k+1$. Then, for each $1 \leq j \leq k+1$, we have $\varphi(j) = c_j$ because one of $i,i'$ does not equal $j$. So we see that $K = K_{\varphi}$, contradicting the assumption that $K$ is colorful. We conclude that $\varphi_1,\dots,\varphi_{k+1}$ are indeed pairwise distinct. It now follows that $\varphi_i(i) \neq c_i$ for each $1 \leq i \leq k + 1$. Indeed, if $\varphi_i(i) = c_i$ then, fixing any $j \in [k+1] \setminus \{i\}$, we have $\varphi_i(\ell) = \varphi_j(\ell)$ for each $\ell \in [k+1] \setminus \{j\}$, contradicting the fact that $K_{\varphi_i}$ and $K_{\varphi_j}$ are edge-disjoint.
		
		Recall that $F'$ consists of vertices $x_1,\dots,x_{k+1},x'_1,\dots,x'_{k+1}$ and copies $F_1,\dots,F_{k+1}$ of $F$ such that the vertex playing the role of $j \in [k+1] \subseteq V(F)$ in $F_i$ is $x_j$ if $j \neq i$ and $x'_j$ if $j = i$ (for every $1 \leq i,j \leq k+1$) and $F_1,\dots,F_{k+1}$ do not intersect outside of $X = \{x_1,\dots,x_{k+1}\}$.
		Now let $\varphi = \varphi_K : V(F') \rightarrow V(H)$ be the function which, for each $1 \leq i \leq k+1$, maps $x_i$ to $c_i$ and agrees with $\varphi_i$ on the vertices of $F_i$ (where we identify $V(F_i)$ with $V(F)$). Then $\varphi(x_i) = c_i$ and $\varphi(x'_i) = \varphi_i(i)$ for each $1 \leq i \leq k+1$.
		It is not hard to see that in order to show that $\varphi$ is an embedding of $F'$ into $H$ it is enough to verify that
		$\text{Im}(\varphi_i) \cap \text{Im}(\varphi_j) = \{c_1,\dots,c_{k+1}\} \setminus \{c_i,c_j\}$ for each $1 \leq i < j \leq k+1$. So fix any $1 \leq i < j \leq k+1$ and consider the set
		$U = U(\varphi_i,\varphi_j) = \{\ell \in V(F) : \varphi_i(\ell) = \varphi_j(\ell)\}$. 
		We need to show that $U = [k+1] \setminus \{i,j\} = A \setminus \{i,j\}$, which would imply that $\text{Im}(\varphi_i) \cap \text{Im}(\varphi_j) = \{c_1,\dots,c_{k+1}\} \setminus \{c_i,c_j\}$, as required. 
		We have $U \cap [k+1] = [k+1] \setminus \{i,j\}$, so 
		$|U \cap A| = k-1$. If 
		$U \neq U \cap [k+1]$, then $U \setminus A \neq \emptyset$,
		so that $\varphi_i$ and $\varphi_j$ are adjacent in $\mathcal{G}$, contradicting the fact that $\varphi_i,\varphi_j \in \mathcal{F}^*$ and that $\mathcal{F}^*$ is an independent set of $\mathcal{G}$. Hence, $U = [k+1] \setminus \{i,j\}$, as required. 
		We have thus shown that each colorful $(k+1)$-clique in $J$ gives rise to a copy of $F'$ in $H$. It is also easy to see that these copies are pairwise distinct. It follows that $H$ contains at least $\delta n^{k+1}$ copies of $F'$.
	
	\subsection{Proof of Lemma \ref{lem:base_case}}\label{subsec:base}
		In the proof of Lemma \ref{lem:base_case}, we will need the following simple claim that can be verified by exhausting all possible cases. The proof is thus omitted.
		\begin{claim}\label{claim:(6,3)}
			Consider the $3$-uniform linear $3$-cycle on vertices $v_1,\dots,v_6$, as depicted in Figure
			1,
			and let $U \subseteq \{v_1,\dots,v_6\}$. Then
			$\Delta(U) \geq
			|U \cap \{v_1,\dots,v_4\}| - \mathds{1}_{\{v_1,\dots,v_4\} \subseteq U}$.
			Moreover, if $U \setminus \{v_1,\dots,v_4\} \neq \emptyset$
			and either $v_1 \notin U$ or $U \cap \{v_2,v_3\} = \emptyset$, then
			$\Delta(U) \geq |U \cap \{v_1,\dots,v_4\}| + 1$.
%
%
		\end{claim}
	
	Let $F$ denote the $3$-uniform linear $3$-cycle (see Figure 1).
	Claim \ref{claim:(6,3)} implies that $F$ satisfies Condition 1 in Definition \ref{def:main} with respect to $A = \{v_1,\dots,v_4\}$. However, $F$ does {\em not} satisfy Condition 2 in that definition, as evidenced, e.g., by the set $U = \{v_1,v_2,v_5\}$. So the ``moreover"-part of Claim \ref{claim:(6,3)} can be thought of as a (non-equivalent) variant of Condition 2 in Definition \ref{def:main}.
	We also note that by going over all possible choices of $A$, one can easily verify that $F$ is {\em not} nice.
	
	\begin{proof}[Proof of Lemma \ref{lem:base_case}]
		Let $F$ be the $3$-graph depicted in Figure 2, having vertices
		$$w_1,w_2,w_3,w_4,w'_1,w'_2,w'_3,w'_4,x_5,x_6,y_5,y_6,z_5,z_6$$ and edges
		\begin{align*}
		&\{w_1,w_2,x_5\},\{x_5,w'_4,x_6\},\{x_6,w_3,w_1\},\{x_5,w_4,y_6\},\{y_6,w'_3,w_1\},
		\\
		&\{w_1,w'_2,y_5\},\{y_5,w_4,x_6\},\{w'_1,w_2,z_5\},\{z_5,w_4,z_6\},\{z_6,w_3,w'_1\}.
		\end{align*}
		Then $v(F) = 14$ and $e(F) = 10$. Solymosi and Solymosi \cite{Solymosi} (implicitly) proved that for every $3$-graph $H$ with $n \geq n_0(r,\varepsilon)$ vertices and at least $\varepsilon n^2$ edges, either $H$ satisfies the assertion of Item 1 in the lemma or $H$ contains at least $\delta(r,\varepsilon) \cdot n^4$ copies of $F$ (with $\delta(r,\varepsilon)$ being roughly $\gamma(3,\varepsilon/r)$, where $\gamma$ is from Theorem \ref{thm:removal}). So, in order to complete the proof, it is enough to show that $F$ is nice.
		
			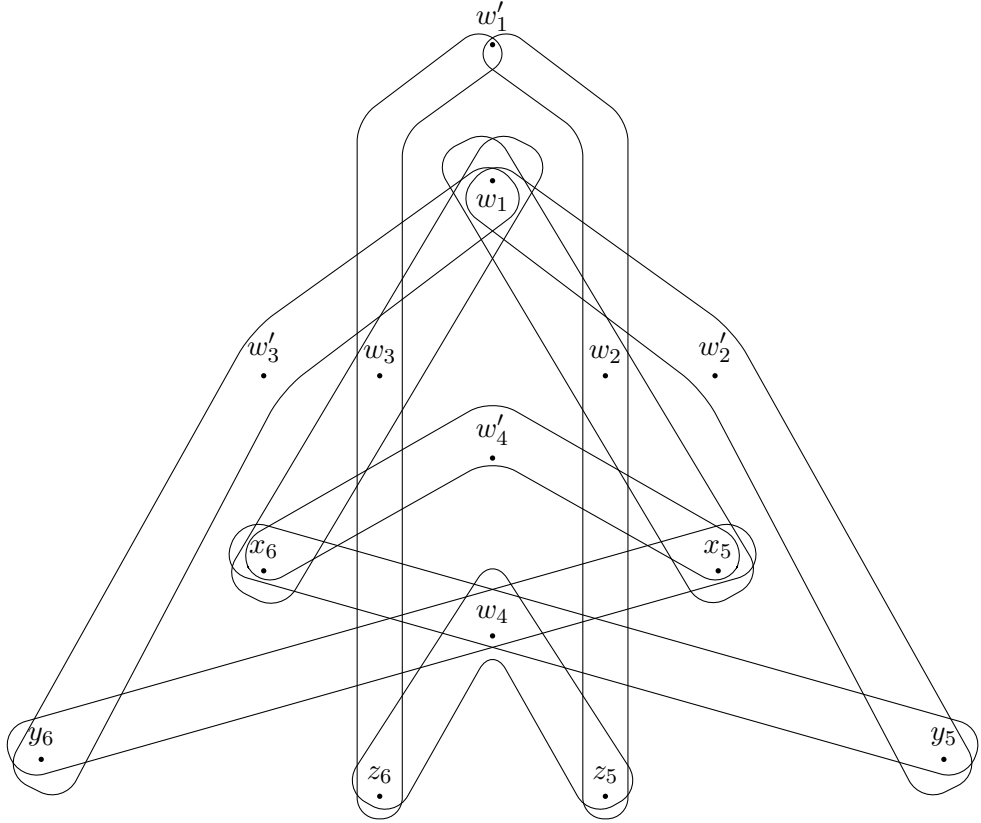
\begin{figure}\label{Fig:(14,10)}
	\centering
	\begin{tikzpicture}
	
	\draw (0,1.73*3) node[fill=black,circle,minimum size=2pt,inner sep=0pt,label=below:$w_1$] {};
	
	\draw (0.5*3,0.865*3) node[fill=black,circle,minimum size=2pt,inner sep=0pt,label=above:$w_2$] {};
	
	
	\draw (-0.5*3,0.865*3) node[fill=black,circle,minimum size=2pt,inner sep=0pt,label=above:$w_3$] {};
	
	\draw (0,1.5) node[fill=black,circle,minimum size=2pt,inner sep=0pt,label=above:$w'_4$] {};
	
	\draw (1*3,0*3) node[fill=black,circle,minimum size=2pt,inner sep=0pt,label=above:$x_5$] {};
	
	\draw (-1*3,0*3) node[fill=black,circle,minimum size=2pt,inner sep=0pt,label=above:$x_6$,left] {};		
	
	\draw (1*3,0.865*3) node[fill=black,circle,minimum size=2pt,inner sep=0pt,label=above:$w'_2$,left] {};	
	
	\draw (-1*3,0.865*3) node[fill=black,circle,minimum size=2pt,inner sep=0pt,label=above:$w'_3$,left] {};
	
	\draw (6,-0.865*2.9) node[fill=black,circle,minimum size=2pt,inner sep=0pt,label=above:$y_5$] {};	
	
	\draw (-6,-0.865*2.9) node[fill=black,circle,minimum size=2pt,inner sep=0pt,label=above:$y_6$] {};
	
	\draw (0,-0.866) node[fill=black,circle,minimum size=2pt,inner sep=0pt,label=above:$w_4$] {};	
	
	\draw (0,7) node[fill=black,circle,minimum size=2pt,inner sep=0pt,label=above:$w'_1$] {};
	
	\draw (1.5,-3) node[fill=black,circle,minimum size=2pt,inner sep=0pt,label=above:$z_5$] {};
	
	\draw (-1.5,-3) node[fill=black,circle,minimum size=2pt,inner sep=0pt,label=above:$z_6$] {};
	
	\draw [rounded corners = 10] (-0.8,5.5) -- (0,5.9) -- (3.6,-0.1) -- (2.8,-0.55) -- cycle;
	
	\draw [rounded corners = 10] (0.8,5.5) -- (0,5.9) -- (-3.6,-0.15) -- (-2.8,-0.55) -- cycle;
	
	\draw [rounded corners = 10] (3.1,-0.25) -- (3.4,0.35) -- (0,2.3) -- (-3.4,0.35) -- (-3.1,-0.25) -- (0,1.5) -- cycle;
	
	\draw [rounded corners = 10] (-0.5,4.9) -- (-0,5.5) -- (3.2,3.2) -- (6.5,-2.7) -- (5.7,-3.1) -- (2.8,2.4) -- cycle;
	
	\draw [rounded corners = 10] (0.5,4.9) -- (0,5.5) -- (-3.2,3.2) -- (-6.5,-2.7) -- (-5.7,-3.1) -- (-2.8,2.4) -- cycle;
	
	\draw [rounded corners = 10] (3.6,0) -- (-6.3,-2.8) -- (-6.55,-2.1) -- (3.35,0.7) -- cycle;
	
	\draw [rounded corners = 10] (-3.6,0) -- (6.3,-2.8) -- (6.55,-2.1) -- (-3.35,0.7) -- cycle;
	
	\draw [rounded corners = 8] (-0.25,6.85) -- (0.15,7.25) -- (1.8,6) -- (1.8,-3.3) -- (1.2,-3.3) -- (1.2,5.8) -- cycle;
	
	\draw [rounded corners = 8] (0.25,6.85) -- (-0.15,7.25) -- (-1.8,6) -- (-1.8,-3.3) -- (-1.2,-3.3) -- (-1.2,5.8) -- cycle;
	
	\draw [rounded corners = 10] (1.3,-3.3) -- (2,-2.9) -- (0,0.2) -- (-2,-2.9) -- (-1.3,-3.3) -- (0,-1) -- cycle;
	\end{tikzpicture}
	\caption{The (14,10)-configuration used in Lemma \ref{lem:base_case}}
\end{figure}
		
		We prove that $F$ satisfies the requirements of Definition \ref{def:main} with $A := \{w_4,w'_1,w'_2,w'_3,w'_4\}$.
		To this end, define $V_1 = \{w'_1,w_2,z_5,w_4,z_6,w_3\}$,
		$V_2 = \{w_1,w'_2,y_5,w_4,x_6,w_3\}$, $V_3 = \{w_1,w_2,x_5,w_4,y_6,w'_3\}$ and $V_4 = \{w_1,w_2,x_5,w'_4,x_6,w_3\}$. Observe that $F[V_i]$ is a linear $3$-cycle for every $1 \leq i \leq 4$. Furthermore, considering the vertex-labeling of the linear $3$-cycle in Figure 1, we see that for each $1 \leq i,j \leq 4$, the role of $v_j$ in $F[V_i]$ is played by $w_j$ if $j \neq i$ and by $w'_j$ if $j = i$.
		Now fix any $U \subseteq V(F)$ and let us show that $U$ satisfies Items 1-2 in Definition \ref{def:main}.
		For each $1 \leq i \leq 4$, define $U_i = U \cap V_i$ and $A_i := (\{w_1,\dots,w_4\} \setminus \{w_i\}) \cup \{w'_i\}$.
		Note that by Claim \ref{claim:(6,3)} we have
		$\Delta(U_i) \geq |U_i \cap A_i| - \mathds{1}_{A_i \subseteq U_i}$.
		
		Let us now express $\Delta(U)$ in terms of $\Delta(U_1),\dots,\Delta(U_4)$.
		It is easy to check that
		\begin{equation}\label{eq:base_case_|U|}
		|U| = \sum_{i = 1}^{4}{|U_i|} -
		2 \cdot |U \cap \{w_1,\dots,w_4\}| - |U \cap \{x_5,x_6\}|
		\end{equation}
		and
		\begin{equation}\label{eq:base_case_e(U)}
		e(U) = \sum_{i=1}^{4}{e(U_i)} - \mathds{1}_{\{w_1,w_2,x_5\} \subseteq U} - \mathds{1}_{\{w_1,w_3,x_6\} \subseteq U}.
		\end{equation}
		Setting $r := \sum_{i = 1}^{4}{\left( \Delta(U_i) - |U_i \cap A_i| \right)}$
		and
		$$
		t :=
		|U \cap \{w_1,w_2,w_3\}| - |U \cap \{x_5,x_6\}| + \mathds{1}_{\{w_1,w_2,x_5\} \subseteq U} + \mathds{1}_{\{w_1,w_3,x_6\} \subseteq U},
		$$
		we combine \eqref{eq:base_case_|U|} and \eqref{eq:base_case_e(U)} to obtain
		\begin{align}\label{eq:Delta(U)_base_case}
		\Delta(U) \nonumber&=
		\sum_{i=1}^{4}{\Delta(U_i)} -
		2 \cdot |U \cap \{w_1,\dots,w_4\}| - |U \cap \{x_5,x_6\}| + \mathds{1}_{\{w_1,w_2,x_5\} \subseteq U} + \mathds{1}_{\{w_1,w_2,x_6\} \subseteq U}
		\\ \nonumber&=
		\sum_{i=1}^{4}{|U_i \cap A_i|} + r - 2 \cdot |U \cap \{w_1,\dots,w_4\}| - |U \cap \{w_1,w_2,w_3\}| + t
		\\ \nonumber&=
		|U \cap A| + r + t.
		\end{align}
		
		To complete the proof, it is enough to show that $r + t \geq -\mathds{1}_{A \subseteq U}$ and that $r + t \geq 1$ if $|U \cap A| \leq 3$ and $U \setminus A \neq \emptyset$. In what follows we will frequently use the fact that
		$\Delta(U_i) \geq |U_i \cap A_i| - \mathds{1}_{A_i \subseteq U_i}$ for each $1 \leq i \leq 4$, as mentioned \nolinebreak above.
		We consider two cases, depending on whether $w_1 \in U$ or not.
		Suppose first that $w_1 \notin U$. In this case we have $t = |U \cap \{w_2,w_3\}| - |U \cap \{x_5,x_6\}|$. Furthermore, $A_i \not\subseteq U_i$ for each $2 \leq i \leq 4$, which implies that $\Delta(U_i) \geq |U_i \cap A_i|$ for these values of $i$.
		Note that if $x_5 \in U$, then $U_i \setminus A_i \neq \emptyset$ for $i = 3,4$, so, by the ``moreover"-part of Claim \ref{claim:(6,3)} (and as $w_1 \notin U$), we have $\Delta(U_i) \geq |U_i \cap A_i| + 1$ for these values of $i$. Similarly, if $x_6 \in U$, then $\Delta(U_i) \geq |U_i \cap A_i| + 1$ for $i = 2,4$.
		Altogether, we conclude that
		$
		r \geq |U \cap \{x_5,x_6\}| + 1 - \mathds{1}_{U \cap \{x_5,x_6\} = \emptyset} - \mathds{1}_{A_1 \subseteq U_1}
		$
		and hence
		\begin{equation}\label{eq:base_case_eq_1}
		r + t \geq |U \cap \{w_2,w_3\}| + 1 - \mathds{1}_{U \cap \{x_5,x_6\} = \emptyset} - \mathds{1}_{A_1 \subseteq U_1}.
		\end{equation}
		If $A_1 \subseteq U_1$, then $\{w_2,w_3\} \subseteq U$ and hence $r + t \geq 1$. So we assume from now on that $A_1 \not\subseteq U_1$. It then easily follows from \eqref{eq:base_case_eq_1} that $r + t \geq 1$ unless
		$U \cap \{w_2,w_3,x_5,x_6\} = \emptyset$. Suppose then that $U \cap \{w_2,w_3,x_5,x_6\} = \emptyset$ and note that in this case $r \geq 0$ and $t = 0$, so in particular $r + t \geq 0 \geq -\mathds{1}_{A \subseteq U}$.
		Furthermore, if $U \setminus A \neq \emptyset$, then
		$U \setminus (A_1 \cup \dots \cup A_4) \neq \emptyset$ (because $U \cap \{w_1,w_2,w_3\} = \emptyset$), so there must be some $1 \leq i \leq 4$ such that $U_i \setminus A_i \neq \emptyset$. Now Claim \ref{claim:(6,3)} implies that $\Delta(U_i) \geq |U_i \cap A_i| + 1$ and hence $r \geq 1$.
		We conclude that if $U \setminus A \neq \emptyset$, then $r + t \geq 1$, as required.
			
		Having handled the case where $w_1 \notin U$, we assume from now on that $w_1 \in U$. Here we consider several subcases, depending on the intersection of $U$ with $\{w_2,w_3\}$. Suppose first that $U \cap \{w_2,w_3\} = \nolinebreak \emptyset$. Then
		$A_i \not\subseteq U_i$ for each $1 \leq i \leq 4$, implying that $r \geq 0$. Furthermore, $t = 1 - |U \cap \{x_5,x_6\}|$.
		So if $U \cap \{x_5,x_6\} = \emptyset$, then $r + t \geq 1$ and we are done.
		On the other hand, if $U \cap \{x_5,x_6\} \neq \emptyset$, then $U_4 \setminus A_4 \neq \emptyset$, which implies, by Claim \ref{claim:(6,3)}, that $\Delta(U_4) \geq |U_4 \cap A_4| + 1$. This shows that $r + t \geq 0 \geq -\mathds{1}_{A \subseteq U}$ and in fact $r + t \geq 1$ if $|U \cap \{x_5,x_6\}| \leq 1$. So from now on we assume that $\{x_5,x_6\} \subseteq U$ and show that $r + t \geq 1$ unless $|U \cap A| \geq 4$. As $\{x_5,x_6\} \subseteq U$, we have $U_i \setminus A_i \neq \emptyset$ for $i = 2,3$. It now follows from Claim \ref{claim:(6,3)} that, for each $i = 2,3$, if $w'_i \notin U$, then $\Delta(U_i) \geq |U_i \cap A_i| + 1$, which, combined with $\Delta(U_4) \geq |U_4 \cap A_4| + 1$, implies that $r \geq 2$ and hence $r + t \geq 1$. So we are done unless $w'_2,w'_3 \in U$. Suppose then that $w'_2,w'_3 \in U$.
		If $w_4 \notin U$, then either $U_2 = \{w_1,w'_2,x_6\}$ or $U_2 = \{w_1,w'_2,y_5,x_6\}$ and in both cases $\Delta(U_2) = 3 = |U_2 \cap A_2| + 1$. But this implies that $r \geq 2$, again giving $r + t \geq 1$. Therefore, we may assume that $w_4 \in U$. Similarly, if $w'_4 \notin U$, then $U_4 = \{w_1,x_5,x_6\}$, from which it follows that $\Delta(U_4) = 3 = |U_4 \cap A_4| + 2$ and hence $r \geq 2$. So we may also assume that $w'_4 \in U$. Altogether, we see that $r + t \geq 1$ unless $\{w'_2,w'_3,w_4,w'_4\} \subseteq U$, which only holds if $|U \cap A| \geq 4$.
		
		Suppose now that $|U \cap \{w_2,w_3\}| = 1$. By symmetry, we may assume without loss of generality that $w_2 \in U$ and $w_3 \notin U$. Then
		$t = 2 - \mathds{1}_{x_6 \in U}$ and $A_i \not\subseteq U_i$ for every $i \in \{1,2,4\}$. It follows that
		$r + t \geq 2 - \mathds{1}_{x_6 \in U} - \mathds{1}_{A_3 \subseteq U_3}$ and hence $r + t \geq 1$ unless $x_6 \in U$ and $A_3 \subseteq U_3$. Suppose then that $x_6 \in U$ and $\{w'_3,w_4\} \subseteq A_3 \subseteq U_3 \subseteq U$. As $x_6 \in U$, we have $U_2 \setminus A_2 \neq \emptyset$. Therefore, if $w'_2 \notin U$, then by Claim \ref{claim:(6,3)} we have $\Delta(U_2) \geq |U_2 \cap A_2| + 1$, which implies that $r \geq 0$ and hence $r + t \geq 1$. So we may assume that $w'_2 \in U$.
		Moreover, if $w'_4 \notin U$, then either $U_4 = \{w_1,w_2,x_6\}$ or $U_4 = \{w_1,w_2,x_5,x_6\}$. Since in both cases $\Delta(U_4) = |U_4 \cap A_4| + 1$, we infer that if $w'_4 \notin U$, then $r \geq 0$ and hence $r + t \geq 1$. Overall, we see that $r + t \geq 1$ unless $\{w'_2,w'_3,w_4,w'_4\} \subseteq U$, as required.
		
		It remains to handle the case where $\{w_2,w_3\} \subseteq U$.
		In this case, we have $t = 3$, so
		$r + t \geq 0$ unless $r = -4$. But if $r = -4$, then $A_i \subseteq U_i$ for each $1 \leq i \leq 4$, which implies that $A \subseteq U$. So we see that
		$r + t \geq -\mathds{1}_{A \subseteq U}$,
		as required.
		Furthermore, if $|U \cap A| \leq 3$, then $\#\{1 \leq i \leq 4 : A_i \subseteq U_i\} \leq 2$ (indeed, if $A_i \subseteq U_i$ for at least $3$ indices $1 \leq i \leq 4$, then $|U \cap \{w'_1,\dots,w'_4\}| \geq 3$ and $w_4 \in U$, implying that $|U \cap A| \geq 4$), so in fact we have $r \geq -2$ and hence $\Delta(U) \geq |U \cap A| + 1$. This completes the proof.
	\end{proof}

\section{Proof of Lemma \ref{lem:Sarkozy_Selkow}}\label{sec:sarkozyselkow}
In this section, we prove Lemma \ref{lem:Sarkozy_Selkow} through a sequence of claims.
We start by defining the $3$-graphs $(G_{\ell})_{\ell \geq 0}$ appearing in the statement of the lemma.
Very roughly speaking, $G_{\ell}$ can be thought of as the $3$-graph obtained by starting with a complete $k$-ary tree of height $\ell$ and replacing each of its vertices by a copy of $G$.

In each of the graphs $G_{\ell}$ ($\ell \geq 0$) we will have a special set $A_{\ell}$ of $k+\ell+1$ vertices, whose role will be similar to the role of the set $A$ in Definition \ref{def:main}. The elements of $A_{\ell} \subseteq V(G_{\ell})$ will be denoted by $x_1,\ldots,x_k,y_0,\ldots,y_{\ell}$. 
If $G^*$ is a copy of some $G_{\ell}$, then
we will use $x_i(G^*)$ and $y_i(G^*)$ to denote the vertices of $G^*$ playing the roles of $x_i$ and $y_i$, respectively, in $G^*$. We will
also write $A_{\ell}(G^*)=\{x_1(G^*),\dots,x_k(G^*),y_0(G^*),\dots,y_{\ell}(G^*)\}$. 
When $G^*$ is clear from the context, we will simply write $A_{\ell},x_1,\dots,x_k,y_0,\dots,y_{\ell}$.

Recall that $G$ is assumed to be nice; so let $A \subseteq V(G)$ be as in Definition \ref{def:main} and note that
$|A| = k+1$ and that $A$ is an independent set. Label the vertices of $A$ (arbitrarily) as $x_1,\dots,x_k,y_0$
and set $G_0$ to be $G$, $y_0(G_0)$ to be $y_0$ and $x_i(G_0)$ to be $x_i$ for every $1 \leq i \leq k$.
In particular, $A_0(G_0)=A$. Proceeding by induction, we fix $\ell \geq 1$ and assume that $G_{\ell-1}$ and the vertices $x_1(G_{\ell-1}),\dots,x_k(G_{\ell-1})$ and $y_0(G_{\ell-1}),\dots,y_{\ell-1}(G_{\ell-1})$ 
(and thus also the set $A_{\ell-1}(G_{\ell-1})$) have already been defined. Now $G_{\ell}$ is defined as follows. Start with the $k+\ell+1$ vertices
$x_1,\ldots,x_k,y_0,\ldots,y_{\ell}$. Define $x_i(G_{\ell})$ to be $x_i$ for every $1 \leq i \leq k$ and $y_i(G_{\ell})$ to be $y_i$
for every $0 \leq i \leq \ell$. In addition to these $k+\ell+1$ vertices, we also have $k$ additional vertices $x'_1,\dots,x'_k$. For each $1 \leq i \leq k$, add a copy of $G_{\ell-1}$, denoted $G_{\ell-1}^i$, in which $x_j$ plays the role of $x_j(G_{\ell-1})$ for each $j \in [k] \setminus \{i\}$, $x'_i$ plays the role of $x_i(G_{\ell-1})$, $y_j$ plays the role of $y_j(G_{\ell-1})$ for each $0 \leq j \leq \ell-1$ and all other
$v(G_{\ell-1}) - k - \ell$ vertices are ``new".
As a last step, add a copy $G^{\ell}$ of $G$ in which $x_i$ plays the role of $x_i(G)$ for each $1 \leq i \leq k$, $y_{\ell}$ plays the role of $y_0(G)$ and all other $v(G) - k - 1$ vertices are ``new". The resulting $3$-graph is \nolinebreak $G_{\ell}$.
	
\begin{claim}\label{claim:Sarkozy_Selkow_basic_facts}
For every $\ell \geq 0$, the set $A_{\ell}(G_{\ell}) \subseteq V(G_{\ell})$ is independent and the graph $G_{\ell}$ satisfies the assertion of Item 1 of Lemma \ref{lem:Sarkozy_Selkow}.
\end{claim}

\begin{proof}
We first prove by induction on $\ell$ that $A_{\ell}(G_{\ell})$ is an independent set. 
For $\ell = 0$, this is guaranteed by our choice of $A_0(G_0) = A$. So fixing $\ell \geq 1$ and assuming the claim holds for $\ell-1$, we now prove it for $\ell$. By the definition of $G_{\ell}$, each edge of $G_{\ell}$ belongs to one of the $3$-graphs $G_{\ell-1}^1,\dots,G_{\ell-1}^k,G^{\ell}$. Moreover, we have 
$V(G_{\ell-1}^i) \cap A_{\ell}(G_{\ell}) \subseteq A_{\ell-1}(G_{\ell-1}^i)$ for every $1 \leq i \leq k$ and $V(G^{\ell}) \cap A_{\ell}(G_{\ell}) = A_0(G^{\ell})$. So the fact that $A_{\ell}(G_{\ell})$ is independent follows from the induction hypothesis for $\ell - 1$ and from the case $\ell = 0$.

Since $A_{\ell}(G_{\ell})$ is independent, the subgraphs 
$G_{\ell-1}^1,\dots,G_{\ell-1}^k,G^{\ell}$, which comprise $G_{\ell}$, are pairwise edge-disjoint.
This implies that $e(G_{\ell}) = k \cdot e(G_{\ell-1}) + e(G)$. 
We now prove the two assertions of Item 1 of the lemma by induction on $\ell$. The case $\ell = 0$ is immediate. 
As for the induction step, observe that for each $\ell \geq 1$, we have
$$
	e(G_{\ell}) = k \cdot e(G_{\ell-1}) + e(G) =
	\left( k \cdot \frac{k^{\ell} - 1}{k - 1} + 1 \right) \cdot e(G) = \frac{k^{\ell+1} - 1}{k - 1} \cdot e(G),
$$
where the second equality follows from the induction hypothesis for $\ell - 1$. 
Moreover, we have
	\begin{align*}
	v(G_{\ell}) &=
	2k + \ell + 1 + k \cdot (v(G_{\ell-1}) - k - \ell) + v(G) - k - 1
	\\ &=
	k + \ell + k \cdot (v(G_{\ell-1}) - k - \ell + 1) + v(G) - k
	\\ &=
	k + \ell + k \cdot e(G_{\ell-1}) + e(G)=
	k + \ell + e(G_{\ell}).
	\end{align*}
Here we used the fact that $\Delta(G) = k$ and the induction hypothesis $\Delta(G_{\ell-1}) = k + \ell - 1$.
So we see that $\Delta(G_{\ell}) = k+\ell$. We have thus proved that Item 1 in Lemma \ref{lem:Sarkozy_Selkow} holds. 
\end{proof}
We now move on to Item 2 of Lemma \ref{lem:Sarkozy_Selkow}. We will prove
this item in the following slightly stronger form, which allows for an inductive proof.

\begin{claim}\label{claim:Sarkozy_Selkow_degeneracy}
		Let $\ell \geq 0$ and $0 \leq t \leq e(G_{\ell})/e(G)$.
		Then there is a subgraph $G'$ of $G_{\ell}$ such that $e(G') = t \cdot e(G)$, 
		$v(G') - e(G') \leq k - 1 + |V(G') \cap \{y_0(G_{\ell}),\dots,y_{\ell}(G_{\ell})\}|$ and 
		$x_1(G_{\ell}),\dots,x_{k-1}(G_{\ell}) \in V(G')$. 
\end{claim}
\noindent
Claim \ref{claim:Sarkozy_Selkow_degeneracy} implies Item 2 of Lemma \ref{lem:Sarkozy_Selkow} because (trivially) $|\{y_0(G_{\ell}),\dots,y_{\ell}(G_{\ell})\}| = \ell+1$. 

\begin{proof}[Proof of Claim \ref{claim:Sarkozy_Selkow_degeneracy}]
		The proof is by induction on $\ell$. 
		The case $t = 0$ is trivial (for every $\ell$), because one can take $G'$ to be the empty graph on $\{x_1(G_{\ell}),\dots,x_{k-1}(G_{\ell})\}$.
		For the base case $\ell = 0$ and for $t = 1$, it is easy to see that the choice $G' := G_0 = G$ satisfies all requirements in the claim. So suppose that $\ell \geq 1$ and let $1 \leq t \leq e(G_{\ell})/e(G)$. 
		Recall that $G_{\ell-1}^1,\dots,G_{\ell-1}^k$ are the copies of $G_{\ell-1}$ which feature in the definition of $G_{\ell}$.
		Let us state the induction hypothesis for $\ell-1$, applied to the copy $G_{\ell-1}^k$ of $G_{\ell-1}$, and plug into it the fact that $y_i(G^k_{\ell-1}) = y_i(G_{\ell}) \text{ for every } 0 \leq i \leq \ell-1 \text{ and } x_i(G^k_{\ell-1}) = x_i(G_{\ell}) \text{ for every } 1 \leq i \leq k-1$, which follows from the definition of $G_{\ell}$.  
		By doing this we get the following statement:
		\paragraph{Induction hypothesis for $G_{\ell-1}^k$:} For every $0 \leq t' \leq e(G_{\ell-1})/e(G)$, there is a subgraph $G''$ of $G_{\ell-1}^k$ which satisfies $e(G'') = t' \cdot e(G)$, $v(G'') - e(G'') \leq k - 1 + |V(G'') \cap \{y_0(G_{\ell}),\dots,y_{\ell-1}(G_{\ell})\}|$ and $x_1(G_{\ell}),\dots,x_{k-1}(G_{\ell}) \in V(G'')$.  
		
		\vspace{3mm}
		For the rest of the proof it will be safe to write $x_i$ for $x_i(G_{\ell})$, $y_i$ for $y_i(G_{\ell})$ and $A_{\ell}$ for $A_{\ell}(G_{\ell})$, as we always consider $G_{\ell}$ and no other hypergraph (so there should be no confusion). This will simplify the notation. 
		
		Suppose first that $t \leq e(G_{\ell-1})/e(G)$. In this case, take $G'$ to be the subgraph $G''$ from the above induction hypothesis with $t' := t$. Then $G'$ satisfies all the requirements in Claim \ref{claim:Sarkozy_Selkow_degeneracy}. 
		
		
Suppose from now on that $t > e(G_{\ell-1})/e(G)$. 
If $t = e(G_{\ell})/e(G)$, then $G' = G_{\ell}$ satisfies all the requirements in Claim \ref{claim:Sarkozy_Selkow_degeneracy}. Suppose then that $t \leq e(G_{\ell})/e(G) - 1 = (k^{\ell+1} - 1)/(k-1) - 1 = k \cdot (k^{\ell} - 1)/(k-1)$, where the first equality uses Item 1 of Lemma \ref{lem:Sarkozy_Selkow}. We also assumed that $t \geq e(G_{\ell-1})/e(G) + 1 = (k^{\ell}-1)/(k-1) + 1$. 
Let $d$ be the unique integer satisfying 	
\begin{equation}\label{eq:choice of d}
d \cdot (k^{\ell} - 1)/(k-1) + 1 \leq t \leq
(d+1) \cdot (k^{\ell} - 1)/(k-1)
\end{equation}
and note that $1 \leq d \leq k-1$ by the above upper and lower bounds on $t$. 
Set
\begin{equation}\label{eq:Sarkozy_Selkow_degeneracy_t'}
t' = t - d \cdot (k^{\ell} - 1)/(k-1) - 1,
\end{equation} 
and note that
$0 \leq t' < (k^{\ell} - 1)/(k-1) = e(G_{\ell-1})/e(G)$ by \eqref{eq:choice of d}. 
Let $G''$ be the subgraph of $G_{\ell-1}^k$ obtained from the above induction hypothesis, so that $e(G'') = t' \cdot e(G)$ and
\begin{equation}\label{eq:degeneracy_G''}
v(G'') - e(G'') \leq k-1 + |V(G'') \cap \{y_0,\dots,y_{\ell-1}\}| \leq |V(G'') \cap A_{\ell}|,
\end{equation}
where the second inequality holds because $y_0,\dots,y_{\ell-1} \in A_{\ell}$ and $x_1,\dots,x_{k-1} \in V(G'') \cap A_{\ell}$ (using the above induction hypothesis). 
Now, take $G'$ to be the subgraph of $G_{\ell}$ consisting of $G_{\ell-1}^1,\dots,G_{\ell-1}^d$, $G^{\ell}$ and $G''$. 
Recall that $G''$ is a subgraph of $G_{\ell-1}^k$ and that $d \leq k-1$. 
This implies that $G_{\ell-1}^1,\dots,G_{\ell-1}^d,G^{\ell},G''$ are pairwise edge-disjoint. Hence, 
\begin{equation}\label{eq:Sarkozy_Selkow_degeneracy_number_of_edges}
e(G') = d \cdot e(G_{\ell-1}) + e(G) + e(G'') =
\left( d \cdot \frac{k^{\ell} - 1}{k - 1} + 1 + t'\right) \cdot e(G) = t \cdot e(G),
\end{equation}
where the second equality uses Item 1 of Lemma \ref{lem:Sarkozy_Selkow} and that $e(G'') = t' \cdot e(G)$, while the last equality uses our choice of $t'$ in \eqref{eq:Sarkozy_Selkow_degeneracy_t'}. 
Next, note that $\{x_1,\dots,x_k,y_0,\dots,y_{\ell}\} = A_{\ell} \subseteq V(G')$ because
$A_{\ell} \setminus \{x_1,y_{\ell}\} \subseteq V(G_{\ell-1}^1) \subseteq V(G')$ (recall that $d \geq 1$)
and $x_1,y_{\ell} \in V(G^{\ell}) \subseteq V(G')$. 
In particular, we have $|V(G') \cap \{y_0,\dots,y_{\ell}\}| = \ell+1$ and $x_1,\dots,x_{k-1} \in V(G')$. 
So to prove that $G'$ satisfies the requirements in Claim \ref{claim:Sarkozy_Selkow_degeneracy}, it remains to show that $v(G') - e(G') \leq k+\ell$. To this end, observe that
\begin{align*}
v(G') &= |A_{\ell}| +
d \cdot (v(G_{\ell-1}) - k - \ell + 1) +
(v(G) - k - 1) + |V(G'') \setminus A_{\ell}| 
\\ &\leq
|A_{\ell}| - 1 +
d \cdot (v(G_{\ell-1}) - k - \ell + 1) +
(v(G) - k) + e(G'') 
\\ &=
k + \ell + d \cdot e(G_{\ell-1}) + e(G) + e(G'') = e(G') + k + \ell,
\end{align*}
where in the first equality we used the definition of $G'$; in the inequality we used the fact that
$|V(G'') \cap A_{\ell}| \geq v(G'') - e(G'')$ by \eqref{eq:degeneracy_G''}; in the second equality we used Item 1 of Lemma \ref{lem:Sarkozy_Selkow} and $|A_{\ell}| = k + \ell + 1$; and in the last equality we used \eqref{eq:Sarkozy_Selkow_degeneracy_number_of_edges}. We have thus shown that $v(G') - e(G') \leq k + \ell$, as required. 
\end{proof}
	
	The rest of this section is devoted to establishing Item 3 of Lemma \ref{lem:Sarkozy_Selkow}.
	To this end, we first prove the following claim, which shows that the niceness of $G$ (with respect to the set $A$) is carried over in a certain sense to $G_{\ell}$ for every $\ell \geq 0$. From now on, we will write $A_{\ell} = \{x_1,\dots,x_k,y_0,\dots,y_{\ell}\}$ (omitting $G_{\ell}$ from the notation). We also set $X :=  \{x_1,\dots,x_k\}$.
	\begin{claim}\label{claim:Sarkozy_Selkow_nice}
		Let $\ell \geq 0$ and let $U \subseteq V(G_{\ell})$ be such that $\{y_0,\dots,y_{\ell-1}\} \subseteq U$. Then
		\begin{enumerate}
			\item $\Delta(U) \geq |U \cap A_{\ell}| - \mathds{1}_{A_{\ell} \subseteq U}$. In particular, if $|U \cap A_{\ell}| \geq k + \ell$, then $\Delta(U) \geq k + \ell$.
			\item If $|U \cap X| \leq k - 2$ and $U \setminus A_{\ell} \neq \emptyset$, then $\Delta(U) \geq |U \cap A_{\ell}| + 1$.
			\item If $|U \cap X| \geq k-1$ and
			$U \cap V(G^{\ell})$ is not contained in $X$, then $\Delta(U) \geq k + \ell$.
		\end{enumerate}
	\end{claim}
	\begin{proof}
		We first prove Items 1-2 by induction on $\ell$ and then deduce Item 3 from Item 1.
		In the base case $\ell = 0$, Items 1-2 immediately follow from the fact that $G_0 = G$ is nice and from our choice of $A_0 = A$ via Definition \ref{def:main}.
		Let now $\ell \geq 1$ and let $U \subseteq V(G_{\ell})$ be such that $\{y_0,\dots,y_{\ell-1}\} \subseteq U$. We start with Item 1.
		For $1 \leq i \leq k$, put $U_i := U \cap V(G_{\ell-1}^i)$. Also, put 
		$U_0 := U \cap V(G^{\ell})$ and note that
		\begin{equation}\label{eq:Sarkozy_Selkow_nice_eq}
		|U \cap A_{\ell}| = |U_0 \cap \{x_1,\dots,x_k,y_{\ell}\}| + \ell
		\end{equation}
		because $y_0,\dots,y_{\ell-1} \in U$ by assumption.
		Since the subgraphs 
		$G_{\ell-1}^1,\dots,G_{\ell-1}^k,G^{\ell}$, which comprise $G_{\ell}$, are pairwise edge-disjoint,
		we have
		$
		e(U) = \sum_{i=0}^{k}{e(U_i)}.
		$
		Observe also that
		$$
		|U| = \sum_{i=0}^{k}{|U_i|} -
		(k-1) \cdot (|U \cap X| + |U \cap \{y_0,\dots,y_{\ell-1}\}| ),
		$$
		as each element of $X \cup \{y_0,\dots,y_{\ell-1}\}$ is contained in exactly $k$ of the sets $V(G_{\ell-1}^1),\dots,V(G_{\ell-1}^k),V(G^{\ell})$ and each of the other vertices of $G_{\ell}$ is contained in exactly one of these sets. 
		From the above formulas for $e(U)$ and $|U|$, it follows that
		\begin{equation}\label{eq:Delta(U)_Sarkozy_Selkow}
			\Delta(U) =
			\sum_{i=0}^{k}{\Delta(U_i)} -
			(k-1) \cdot ( |U \cap X| + \ell ).
		\end{equation}
		Here we used the fact that $\{y_0,\dots,y_{\ell-1}\} \subseteq U$ by assumption.
		Recall that by the definition of $G_{\ell}$, for each $1 \leq i \leq k$, we have 
		$$A_{\ell-1}(G_{\ell-1}^i) = \{x_1,\dots,x_k,y_0,\dots,y_{\ell-1},x'_i\} \setminus \{x_i\}.$$ 
		By the induction hypothesis for $\ell - 1$, applied to the copy $G_{\ell-1}^i$ of $G_{\ell-1}$, we get
		\begin{equation}\label{eq:Delta(U_i)_Sarkozy_Selkow}
		\Delta(U_i) \geq
		|U_i \cap A_{\ell-1}(G_{\ell-1}^i)| - \mathds{1}_{A_{\ell-1}(G_{\ell-1}^i) \subseteq U_i} \geq
		|U_i \cap (A_{\ell} \setminus \{x_i,y_{\ell}\})|,
		\end{equation}
		where the second inequality follows by considering whether $x'_i \in U_i$ or not. 
		From \eqref{eq:Delta(U_i)_Sarkozy_Selkow}, we obtain
		\begin{equation}\label{eq:Delta(U_1)+...+Delta(U_k)_Sarkozy_Selkow}
		\begin{split}
		\sum_{i=1}^{k}{\Delta(U_i)} &\geq
		\sum_{i=1}^{k}{|U_i \cap (A_{\ell} \setminus \{x_i,y_{\ell}\})|} \\ &=
		(k-1) \cdot |U \cap X| + k \cdot |U \cap \{y_0,\dots,y_{\ell-1}\}| \\ &=
		(k-1) \cdot |U \cap X| + k\ell,
		\end{split}
		\end{equation}
		where in the first equality we used the fact that each element of $X$ belongs to exactly $k - 1$ of the sets $A_{\ell} \setminus \{x_i,y_{\ell}\}$ (where $1 \leq i \leq k$) and each element of $\{y_0,\dots,y_{\ell-1}\}$ belongs to all of these sets.
		Plugging the above into \eqref{eq:Delta(U)_Sarkozy_Selkow} gives
		\begin{equation}\label{eq:Delta(U)_Sarkozy_Selkow_1}
		\Delta(U) \geq
		\Delta(U_0) + \ell.
		\end{equation}
		Since $G$ is nice and $G^{\ell}$ is a copy of $G$ in which $y_{\ell}$ plays the role of $y_0(G)$, we have
		\begin{equation}\label{eq:Delta(U_0)_Sarkozy_Selkow}
		\Delta(U_0) \geq |U_0 \cap \{x_1,\dots,x_k,y_{\ell}\}| - \mathds{1}_{\{x_1,\dots,x_k,y_{\ell}\} \subseteq U_{0}}.
		\end{equation}
		Note that $\mathds{1}_{\{x_1,\dots,x_k,y_{\ell}\} \subseteq U_0} = \mathds{1}_{A_{\ell} \subseteq U}$ because $y_0,\dots,y_{\ell-1} \in U$. 
		By combining \eqref{eq:Sarkozy_Selkow_nice_eq}, \eqref{eq:Delta(U)_Sarkozy_Selkow_1}, \eqref{eq:Delta(U_0)_Sarkozy_Selkow}, we get
		$$
		\Delta(U) \geq
		\Delta(U_0) + \ell
		\geq
		|U_0 \cap \{x_1,\dots,x_k,y_{\ell}\}| - \mathds{1}_{\{x_1,\dots,x_k,y_{\ell}\} \subseteq U_0} + \ell =
		|U \cap A_{\ell}| - \mathds{1}_{A_{\ell} \subseteq U},
		$$
		thus establishing Item 1.
		
		Next, we prove Item 2. Suppose then that $|U \cap X| \leq k - 2$ and $U \setminus A_{\ell} \neq \emptyset$. The inequality $|U \cap X| \leq k - 2$ implies that
		$|U_0 \cap \{x_1,\dots,x_k,y_{\ell}\}| \leq k-1$ and that $A_{\ell-1}(G_{\ell-1}^i) \not\subseteq U_i$ for each $1 \leq i \leq k$.
		Since $U \setminus A_{\ell} \neq \emptyset$, there is $0 \leq i \leq k$ such that $U_i \setminus A_{\ell} \neq \emptyset$. Suppose first that $i = 0$. Then $U_0 \setminus \{x_1,\dots,x_k,y_{\ell}\} \neq \emptyset$, which, combined with $|U_0 \cap \{x_1,\dots,x_k,y_{\ell}\}| \leq k-1$, implies that
		$\Delta(U_0) \geq |U_0 \cap \{x_1,\dots,x_k,y_{\ell}\}| + 1$. Here we used the niceness of $G$ (see Item 2 in Definition \ref{def:main}).
		By plugging our bound on $\Delta(U_0)$ into  \eqref{eq:Delta(U)_Sarkozy_Selkow_1} and using \eqref{eq:Sarkozy_Selkow_nice_eq},
		we get
		$\Delta(U) \geq \Delta(U_0) + \ell \geq |U_0 \cap \{x_1,\dots,x_k,y_{\ell}\}| + 1 + \ell =
		|U \cap A_{\ell}| + 1$, as required.
		Now suppose that $1 \leq i \leq k$.
		We claim \nolinebreak that
		\begin{equation}\label{eq:Delta(U_i)_Sarkozy_Selkow_1}
		\Delta(U_i) \geq |U_i \cap (A_{\ell} \setminus \{x_i,y_{\ell}\})| + 1.
		\end{equation}
		In other words, we show that the inequality bounding the leftmost term in \eqref{eq:Delta(U_i)_Sarkozy_Selkow} by the rightmost one is strict.
		If $x'_i \in U_i$, then
		$$
		\Delta(U_i) \geq |U_i \cap A_{\ell-1}(G_{\ell-1}^i)| - \mathds{1}_{A_{\ell-1}(G_{\ell-1}^i) \subseteq U_i} = 
		|U_i \cap A_{\ell-1}(G_{\ell-1}^i)| \geq |U_i \cap (A_{\ell} \setminus \{x_i,y_{\ell}\})| + 1,$$ as required. Here, in the first inequality we used \eqref{eq:Delta(U_i)_Sarkozy_Selkow}, in the equality we used the fact that $A_{\ell-1}(G_{\ell-1}^i) \not\subseteq U_i$ (as mentioned above) and in the last inequality we used the fact that 
		\linebreak $x'_i \in A_{\ell-1}(G_{\ell-1}^i) \setminus A_{\ell}$. So suppose now
		 that $x'_i \notin U_i$ and note that in this case
		$U_i \setminus A_{\ell-1}(G_{\ell-1}^i) \neq \emptyset$ because $U_i \setminus A_{\ell} \neq \emptyset$ and $A_{\ell-1}(G_{\ell-1}^i) \subseteq A_{\ell} \cup \{x'_i\}$.
		Moreover, the intersection of $U_i$ with the set $\{x_1(G_{\ell-1}^i),\dots,x_k(G_{\ell-1}^i)\} = \{ x_1,\dots,x_k,x'_i \} \setminus \{x_i\}$ is of size at most $k - 2$, because $|U \cap X| \leq k-2$ and $x'_i \notin U_i$. 	
		So, by the induction hypothesis applied to the copy $G_{\ell-1}^i$ of $G_{\ell-1}$, we have
		$$
		\Delta(U_i) \geq |U_i \cap A_{\ell-1}(G_{\ell-1}^i)| + 1 \geq
		|U_i \cap (A_{\ell} \setminus \{x_i,y_{\ell}\})| + 1,
		$$
		where the last inequality uses that $A_{\ell} \setminus \{x_i,y_{\ell}\} \subseteq A_{\ell-1}(G_{\ell-1}^i)$. This proves \eqref{eq:Delta(U_i)_Sarkozy_Selkow_1}.
		By repeating the calculation in \eqref{eq:Delta(U_1)+...+Delta(U_k)_Sarkozy_Selkow} and plugging in
		\eqref{eq:Delta(U_i)_Sarkozy_Selkow_1} and \eqref{eq:Delta(U_i)_Sarkozy_Selkow} (which we use for each $j \in [k] \setminus \{i\}$), we obtain
		\begin{equation*}
		\sum_{i=1}^{k}{\Delta(U_i)} \geq
		|U_i \cap (A_{\ell} \setminus \{x_i,y_{\ell}\})| + 1 + 
		\sum_{j \in [k] \setminus \{i\}}{|U_j \cap (A_{\ell} \setminus \{x_j,y_{\ell}\})|} =
		(k-1) \cdot |U \cap X| + k\ell + 1.
		\end{equation*}
		By plugging the above into \eqref{eq:Delta(U)_Sarkozy_Selkow}, we get
		\begin{align*}
		\Delta(U) &= \sum_{i=0}^{k}{\Delta(U_i)} - (k-1) \cdot (|U \cap X| + \ell)
		\geq
		\Delta(U_0) + \ell + 1
		\\ &\geq
		|U_0 \cap \{x_1,\dots,x_k,y_{\ell}\}| - \mathds{1}_{\{x_1,\dots,x_k,y_{\ell}\} \subseteq U_{0}} + \ell + 1 \\ &=
		|U_0 \cap \{x_1,\dots,x_k,y_{\ell}\}| + \ell + 1 =
		|U \cap A_{\ell}| + 1,
		\end{align*} 
		where the inequality uses \eqref{eq:Delta(U_0)_Sarkozy_Selkow}, the penultimate equality holds because $\{x_1,\dots,x_k,y_{\ell}\} \not\subset U_0$ as $|U \cap X| \leq k-2$ and the last equality uses \eqref{eq:Sarkozy_Selkow_nice_eq}. 
		This completes the inductive proof of Items 1-2.
		
		It remains to derive Item 3 from Item 1. Suppose then that $|U \cap X| \geq k-1$ and that $U_0 \not\subseteq X$.
		If $X \subseteq U$ or $y_{\ell} \in U$, then $|U \cap A_{\ell}| \geq k + \ell$, in which case Item 1 gives $\Delta(U) \geq k + \ell$, as required.
		So we may assume that $|U \cap X| = k-1$ and $y_{\ell} \notin U$. Since $U_0$ is not contained in $X$, we must have $U_0 \setminus \{x_1,\dots,x_k,y_{\ell}\} \neq \emptyset$. So by the niceness of $G$ we have
		$\Delta(U_0) \geq |U_0 \cap \{x_1,\dots,x_k,y_{\ell}\}| + 1 = k$. Plugging this into \eqref{eq:Delta(U)_Sarkozy_Selkow_1} gives $\Delta(U) \geq k + \ell$, as required.
	\end{proof}

	Item 3 of Lemma \ref{lem:Sarkozy_Selkow} will be derived from the following claim, in a manner similar to the derivation of Lemma \ref{lem:main} from Lemma \ref{lem:super_exponential}. 
	\begin{claim}\label{claim:Sarkozy_Selkow_main}
	For every $\ell \geq 0$, $r \geq 0$ and $\varepsilon \in (0,1)$, there are $\delta = \delta_{\ref{claim:Sarkozy_Selkow_main}}(\ell,r,\varepsilon) \in (0,1)$ and $n_0 = n_0(\ell,r,\varepsilon)$ such that, for every $3$-graph $H$ on $n \geq n_0$ vertices, if $H$ contains at least $\varepsilon n^{k + \ell}$ copies of $G_{\ell}$, then (at least) one of the following conditions is satisfied:
		\begin{enumerate}
			\item There is $0 \leq q \leq e(G_{\ell})-1$ such that, for every $1 \leq i \leq r$, the $3$-graph $H$ contains a $(v',e')$-configuration which contains a copy of $G_{\ell}$ and satisfies
			$v' - e' \leq k + \ell$ and
			$e' = q + i \cdot (e(G_{\ell}) - q)$.
			\item $H$ contains at least $\delta \cdot n^{k + \ell + 1}$ copies of $G_{\ell+1}$.
		\end{enumerate}
	\end{claim}
	\begin{proof}
	We proceed similarly to the proof of Lemma \ref{lem:super_exponential}.
	Fixing $\ell \geq 0$, we set $v := v(G_{\ell})$,
	$$
	\zeta := 2^{-v (1 + 2^{v}r)} \cdot v^{-v} \cdot \varepsilon,
	$$
	$
	\delta = \delta(\ell,r,\varepsilon) =
	\frac{\zeta}{4} \cdot
	\gamma\left( k, \frac{\zeta}{2} \right)
	$
	and
	$n_0 = n_0(\ell,r,\varepsilon) = \frac{2}{\gamma(k,\frac{\zeta}{2})}$,
	where $\gamma$ is from Theorem \ref{thm:removal}.
	
	Let $H$ be a $3$-graph on $n \geq n_0$ vertices which contains at least $\varepsilon n^{k + \ell}$ copies of $G_{\ell}$.
	Partition the vertices of $H$ randomly into sets $(C_z : z \in V(G_{\ell}))$ by choosing, for each vertex $x \in V(H)$, a vertex $z \in V(G_{\ell})$ uniformly at random and independently (of the choices made for all other vertices) and placing $x$ in part $C_z$.
	A copy of $G_{\ell}$ in $H$ will be called {\em good} if, for each $z \in V(G_{\ell})$, the vertex playing the role of $z$ in this copy belongs to $C_z$.
	Since $H$ contains at least $\varepsilon n^{k+\ell}$ copies of $G_{\ell}$, there are in expectation at least $v^{-v} \cdot \varepsilon n^{k+\ell}$ good copies of $G_{\ell}$. So fix a partition $(C_z : z \in V(G_{\ell}))$ with at least this number of good copies of $G_{\ell}$ and denote the set of these copies by $\mathcal{F}$. 
	We will identify each good copy of $G_{\ell}$ with the corresponding embedding
	$\varphi : V(G_{\ell}) \rightarrow V(H)$, while noting that $\varphi(z) \in C_z$ for each $z \in V(G_{\ell})$.
	Recall that $G^{\ell}$ is the copy of $G$ featured in the definition of $G_{\ell}$. 
	Define an auxiliary graph $\mathcal{G}$ on $\mathcal{F}$ as follows. For each pair of distinct $\varphi_1,\varphi_2 \in \mathcal{F}$, we set
	$
	U(\varphi_1,\varphi_2) := \{z \in V(G_{\ell}) : \varphi_1(z) = \varphi_2(z)\}
	$
	and let $\{\varphi_1,\varphi_2\}$ be an edge in $\mathcal{G}$ if and only if
	$U := U(\varphi_1,\varphi_2)$ satisfies $\{y_0,\dots,y_{\ell-1}\} \subseteq U$, as well as (at least) one of the following three conditions:
	\begin{enumerate}
		\item[(i)] $|U \cap A_{\ell}| \geq k + \ell$.
		\item[(ii)] $y_{\ell} \in U$ and either $|U \cap X| \geq k-1$ or $|U \cap X| = k-2$ and $U \setminus A_{\ell} \neq \emptyset$.
		\item[(iii)]
		$|U \cap X| \geq k-1$ and $U \cap V(G^{\ell})$ is not contained in $X$.
		\end{enumerate}
	
	Suppose first that there is $\varphi \in \mathcal{F}$ whose degree in $\mathcal{G}$ is at least
	$$
	d := 2^{v (1 + 2^{v}r)} .
	$$
	We show that in this case, Item 1 in Claim \ref{claim:Sarkozy_Selkow_main} holds. 
	Let $\varphi_1,\dots,\varphi_d$ be distinct neighbours of $\varphi$ in $\mathcal{G}$.
	By the pigeonhole principle, there is $I' \subseteq [d]$ of size at least $2^{-v}d = 2^{v 2^{v}r}$ and a set $U' \subseteq V(G_{\ell})$ such that, for all $i \in I'$, it holds that $U(\varphi,\varphi_i) = U'$.
	As in the proof of Lemma \ref{lem:super_exponential}, we consider the coloring
	$\{i,j\} \mapsto U(\varphi_i,\varphi_j)$ of the pairs $\{i,j\} \in \binom{I'}{2}$ and use a well-known bound for multicolor Ramsey numbers \cite{CFS} to obtain
	a set $I \subseteq I'$ of size $|I| = r$ and a set $U \subseteq V(G_{\ell})$ such that $U(\varphi_i,\varphi_j) = U$ for all $\{i,j\} \in \binom{I}{2}$. Observe that for each $\{i,j\} \in \binom{I}{2}$, we have $U \supseteq U(\varphi,\varphi_i) \cap U(\varphi,\varphi_j) = U'$. In particular, $\{y_0,\dots,y_{\ell-1}\} \subseteq U' \subseteq U$ (by the definition of $\mathcal{G}$). Note also that $U \neq V(G_{\ell})$ because the copies $(\varphi_i : i \in I)$ of $G_{\ell}$ are distinct.
	
	We now use Claim \ref{claim:Sarkozy_Selkow_nice} to prove that $\Delta(U) \geq k + \ell$. 
	The definition of the graph $\mathcal{G}$ implies that the set $U'$ must satisfy one of the conditions (i)-(iii) above. Note that for each of these three conditions, if it is satisfied by $U'$, then it is also satisfied by every superset of $U'$ and, in particular, by $U$. Now, if $U$ satisfies Condition (i) (resp. (iii)), then the bound $\Delta(U) \geq k + \ell$ immediately follows from Item 1 (resp. 3) of Claim \ref{claim:Sarkozy_Selkow_nice}. Suppose then that $U$ satisfies Condition (ii). If $|U \cap X| \geq k-1$, then $|U \cap A_{\ell}| \geq k + \ell$ (since $y_0,\dots,y_{\ell-1} \in U$ and Condition (ii) supposes that $y_{\ell} \in U$), so again we can use Item 1 of Claim \ref{claim:Sarkozy_Selkow_nice}. Finally, if $|U \cap X| = k - 2$ and $U \setminus A_{\ell} \neq \emptyset$, then we have 
	$\Delta(U) \geq |U \cap A_{\ell}| + 1 = k + \ell$, where the inequality is Item 2 of Claim \ref{claim:Sarkozy_Selkow_nice} and the equality holds because $\{y_0,\dots,y_{\ell}\} \subseteq U$ and $|U \cap X| = k - 2$. We have thus shown that $\Delta(U) \geq k + \ell$ in all cases. 

	Suppose without loss of generality that $I = [r]$. Put
	$W := \varphi_1(U) = \dots = \varphi_r(U)$ and denote
	$V_i := \varphi_i(V(G_{\ell}) \setminus U) \subseteq V(H)$ for each
	$1 \leq i \leq r$. Note that $V_1,\dots,V_r$ are pairwise disjoint.
	Now, fix any $1 \leq i \leq r$ and set $V := V_1 \cup \dots \cup V_i \cup W$. Then
	$$|V| = |U| + i \cdot (v(G_{\ell}) - |U|) = i \cdot v(G_{\ell}) - (i-1) \cdot |U|$$ and
	$$
	e_H(V) \geq e(U) + i \cdot (e(G_{\ell}) - e(U)) =
	i \cdot e(G_{\ell}) - (i-1) \cdot e(U).
	$$
	It follows that
	\begin{align*}
	|V| - e_{H}(V) &\leq i \cdot (v(G_{\ell}) - e(G_{\ell})) - (i-1) \cdot (|U| - e(U))
	=
	i \cdot (k + \ell) - (i-1) \cdot \Delta(U)
	\\ &\leq i \cdot (k + \ell) - (i-1) \cdot (k + \ell) = k + \ell.
	\end{align*}
	Moreover, it is evident that $H[V]$ contains a copy of $G_{\ell}$. Finally,
	note that 
	$e(U) = |U| - \Delta(U) \leq |U| - (k + \ell) \leq v(G_{\ell}) - 1 - (k + \ell) = e(G_{\ell}) - 1$, where the second inequality holds because $U \neq V(G_{\ell})$.
	Combining all of the above, we see that the assertion of Item 1 in the claim holds with $q := e(U)$.
	This completes the proof in the case where $\mathcal{G}$ has a vertex of degree at least $d$.
	
	From now on we assume that the maximum degree of $\mathcal{G}$ is strictly smaller than $d$. Let $\mathcal{F}^* \subseteq \mathcal{F}$ be an independent set in $\mathcal{G}$ of size at least $v(\mathcal{G})/d = |\mathcal{F}|/d$. For each $\ell$-tuple of vertices
	$u = (u_0,\dots,u_{\ell-1}) \in \tilde{C} := C_{y_0} \times \dots \times C_{y_{\ell-1}}$, we denote by $\mathcal{F}^*(u)$ the set of all $\varphi \in \mathcal{F}^*$ such that $\varphi(y_i) = u_i$ for each $0 \leq i \leq \ell-1$. Note that
	\begin{equation}\label{eq:Sarkozy_Selkow_embedding_count}
	\sum_{u \in \tilde{C}}{|\mathcal{F}^*(u)|} = |\mathcal{F}^*| \geq \frac{|\mathcal{F}|}{d} \geq
	\frac{\varepsilon n^{k+\ell}}{v^v d} =
	\zeta n^{k + \ell} \; .
	\end{equation}
	
	We claim that $|\mathcal{F}^*(u)| \leq n^k$ for each $u \in \tilde{C}$. To see this, fix any such $u$ and let $\varphi_1,\varphi_2 \in \mathcal{F}^*(u)$ be distinct. If $\varphi_1(x_i) = \varphi_2(x_i)$ for each $1 \leq i \leq k$, then $\{x_1,\dots,x_k,y_0,\dots,y_{\ell-1}\} \subseteq U(\varphi_1,\varphi_2)$. But then $U$ satisfies Condition (i) above, implying that $\{\varphi_1,\varphi_2\} \in E(\mathcal{G})$, contradicting the fact that $\mathcal{F}^*$ is an independent set in $\mathcal{G}$. So we see that for each $u \in \tilde{C}$ and each $\varphi \in \mathcal{F}^*(u)$, the values of $\varphi(x_1),\dots,\varphi(x_k)$ determine $\varphi$ uniquely. It follows that indeed $|\mathcal{F}^*(u)| \leq n^k$.
	Now, by using \eqref{eq:Sarkozy_Selkow_embedding_count} and averaging, we get that there are at least $\frac{\zeta}{2} n^{\ell}$ tuples $u \in \tilde{C}$ which satisfy $|\mathcal{F}^*(u)| \geq \frac{\zeta}{2}n^k$. Let $C \subseteq \tilde{C}$ be the set of all such tuples $u$. We will show that for every $u = (u_0,\dots,u_{\ell-1})\in C$, there are at least $\frac{1}{2}\gamma(k,\frac{\zeta}{2}) \cdot n^{k+1}$ copies of $G_{\ell+1}$ in $H$ in which $u_i$ plays the role of $y_i(G_{\ell+1})$ for every 
	$0 \leq i \leq \ell - 1$. Combining this with the fact that $|C| \geq \frac{\zeta}{2}n^{\ell}$, we will conclude that $H$ contains at least
	$\frac{\zeta}{2}n^{\ell} \cdot \frac{1}{2}\gamma(k,\frac{\zeta}{2}) \cdot n^{k+1} = \delta n^{k + \ell + 1}$ copies of $G_{\ell + 1}$, as required.

	So fix any $u \in C$. We define an auxiliary $k$-uniform $(k+1)$-partite hypergraph $J(u)$ with parts $C_{x_1},\dots,C_{x_k},C_{y_{\ell}}$, as follows. For each $\varphi \in \mathcal{F}^*(u)$, put a $k$-uniform $(k+1)$-clique in $J(u)$ on the vertices $\varphi(x_1) \in C_{x_1},\dots,\varphi(x_k) \in C_{x_k}, \; \varphi(y_{\ell}) \in C_{y_{\ell}}$. We denote this clique by $K_{\varphi}$.
	We claim that the cliques $(K_{\varphi} : \varphi \in \mathcal{F}^*(u))$ are pairwise edge-disjoint. To this end, fix any pair of distinct $\varphi_1,\varphi_2 \in \mathcal{F}^*(u)$ and suppose, for the sake of contradiction, that the cliques $K_{\varphi_1},K_{\varphi_2}$ share an edge. Then there is $Z \subseteq \{x_1,\dots,x_k,y_{\ell}\}$ of size $|Z| = k$ such that
	$\varphi_1(z) = \varphi_2(z)$ for every $z \in Z$. It follows that
	$Z \cup \{y_0,\dots,y_{\ell-1}\} \subseteq U(\varphi_1,\varphi_2)$, where $\{y_0,\dots,y_{\ell-1}\} \subseteq U(\varphi_1,\varphi_2)$ holds because $\varphi_1,\varphi_2 \in \mathcal{F}^*(u)$. Therefore,  $|U(\varphi_1,\varphi_2) \cap A_{\ell}| \geq k + \ell$, implying that  $U(\varphi_1,\varphi_2)$ satisfies Condition (i) above. But this implies that $\{\varphi_1,\varphi_2\} \in E(\mathcal{G})$, which contradicts the fact that $\mathcal{F}^*(u) \subseteq \mathcal{F}(u)$ is an independent set in $\mathcal{G}$.
	We have thus shown that the cliques $(K_{\varphi} : \varphi \in \mathcal{F}^*(u))$ are indeed pairwise edge-disjoint. 
	
	It follows from the previous paragraph that $J(u)$ contains a collection of
	$
	|\mathcal{F}^*(u)| \geq \frac{\zeta}{2} n^k
	$
	pairwise edge-disjoint $(k+1)$-cliques. By Theorem \ref{thm:removal}, the number of $(k+1)$-cliques in $J(u)$ is at least $\gamma(k,\frac{\zeta}{2}) \cdot n^{k+1}$. A $(k+1)$-clique $K$ in $J(u)$ is called {\em colorful} if it is not equal to $K_{\varphi}$ for any $\varphi \in \mathcal{F}^*(u)$.
	Since there are at most $|\mathcal{F}^*(u)| \leq n^{k}$ non-colorful $(k+1)$-cliques, the number of colorful $(k+1)$-cliques in $J(u)$ is at least
	$\gamma(k,\frac{\zeta}{2}) \cdot n^{k+1} - n^{k} \geq \frac{1}{2} \gamma(k,\frac{\zeta}{2}) \cdot n^{k+1}$ (here we use our choice of $n_0$).
	
	To complete the proof, it remains to show that each colorful $(k+1)$-clique in $J(u)$ corresponds to a copy of $G_{\ell+1}$ in $H$.
	Fix any colorful $(k+1)$-clique $K = \{w_1,\dots,w_k,u_{\ell}\}$, where $u_{\ell}$ is the unique vertex of $K$ contained in $C_{y_{\ell}}$ and, for each $1 \leq i \leq k$, $w_i$ is the unique vertex of $K$ contained in $C_{x_i}$. By the definition of $J(u)$, each of the $k+1$ edges of $K$ corresponds to an embedding of $G_{\ell}$ into $H$. More precisely, there are $\varphi_0,\varphi_1,\dots,\varphi_k \in \mathcal{F}^*(u)$ such that:
	\begin{itemize}
		\item For each $1 \leq i \leq k$, $\varphi_i(y_{\ell}) = u_{\ell}$ and $\varphi_i(x_j) = w_j$ for each $j \in [k] \setminus \{i\}$.
		\item $\varphi_0(x_i) = w_i$ for each $1 \leq i \leq k$.
	\end{itemize}
	We claim that $\varphi_0,\dots,\varphi_{k}$ are pairwise distinct. 
	Assume, for the sake of contradiction, that $\varphi_i = \varphi_{i'} =: \varphi$ for some $0 \leq i < i' \leq k$. 
	Then $\varphi(x_j) = w_j$ for each $1 \leq j \leq k$ because of the two items above, as (evidently) one of $i,i'$ does not equal $j$.  Similarly, since $i,i'$ cannot both equal $0$, the first item above implies that $\varphi(y_{\ell}) = u_{\ell}$. We now see that $K = K_{\varphi}$, contradicting the assumption that $K$ is colorful. Hence, $\varphi_0,\dots,\varphi_{k}$ are indeed pairwise distinct. 
	Now the edge-disjointness of the cliques $K_{\varphi_0},K_{\varphi_1},\dots,K_{\varphi_k}$ implies that $w'_i := \varphi_i(x_i) \neq w_i$ for each $1 \leq i \leq k$ and that $u_{\ell+1} := \varphi_0(y_{\ell}) \neq u_{\ell}$.
	
	We now show how to construct a copy of $G_{\ell+1}$ using the copies of $G_{\ell}$ corresponding to $\varphi_1,\dots,\varphi_k$ and the copy of $G$ corresponding to $\varphi_0(G^{\ell})$. In this copy of $G_{\ell+1}$, the role of $x_i(G_{\ell+1})$ will be played by $w_i$ for every $1 \leq i \leq k$, the role of the vertex $x'_i \in V(G_{\ell+1})$ will be played by $w'_i$ for every $1 \leq i \leq k$ (recall the definition of $G_{\ell+1}$) and the role of $y_i(G_{\ell+1})$ will be played by $u_i$ for every $0 \leq i \leq \ell+1$. (Recall that the vertices $u_0,\dots,u_{\ell-1}$ have already been fixed via the choice of $u$.) Note that, for each $1 \leq i \leq k$, the embedding $\varphi_i$ corresponds to a copy of $G_{\ell}$ in which $w_j$ plays the role of $x_j(G_{\ell})$ for every $j \in [k] \setminus \{i\}$, $w'_i$ plays the role of $x_i(G_{\ell})$ and $u_j$ plays the role of $y_j(G_{\ell})$ for every $0 \leq j \leq \ell$. This copy of $G_{\ell}$ will play the role of $G_{\ell}^i$ in our copy of $G_{\ell+1}$. Similarly, restricting $\varphi_0$ to $V(G^{\ell})$ gives a copy of $G$ in which $w_i$ plays the role of $x_i(G)$ for each $1 \leq i \leq k$ and $u_{\ell+1}$ plays the role of $y_0(G)$ (as $y_0(G^{\ell}) = y_{\ell}(G_{\ell})$ and
	$\varphi_0(y_{\ell}(G_{\ell})) = u_{\ell+1}$).
	By the definition of $G_{\ell+1}$, in order to show that \linebreak$\text{Im}(\varphi_1) \cup \dots \cup \text{Im}(\varphi_k) \cup \varphi_0(V(G^{\ell}))$ spans a copy of $G_{\ell+1}$, it suffices to verify that the $k$ copies of $G_{\ell}$ given by $\varphi_1,\dots,\varphi_k$, and the copy of $G$ given by $\varphi_0(G^{\ell})$, do not intersect outside of 
	$\{w_1,\dots,w_k,u_0,\dots,u_{\ell}\}$. Therefore, our goal is to show that
	$\text{Im}(\varphi_i) \cap \nolinebreak \text{Im}(\varphi_j) = \{w_1,\dots,w_k,u_0,\dots,u_{\ell}\} \setminus \{w_i,w_j\}$ for each $1 \leq i < j \leq k$ and that
	$\text{Im}(\varphi_i) \cap \varphi_0(V(G^{\ell})) = \{w_1,\dots,w_k\} \setminus \{w_i\}$ for each $1 \leq i \leq k$. We start with the former statement.
	Fix any $1 \leq i < j \leq k$. Setting $U := U(\varphi_i,\varphi_j)$,
	note that $\text{Im}(\varphi_i) \cap \text{Im}(\varphi_j) = \varphi_i(U) = \varphi_j(U)$, that $y_0,\dots,y_{\ell} \in U$ and that $U \cap X = X \setminus \{x_i,x_j\}$. Indeed, by the two items above we know that $\varphi_i,\varphi_j$ agree on $y_{\ell}$ and on $X \setminus \{x_i,x_j\}$ and, on the other hand, that $\varphi_i(x_j) = w_j \neq \varphi_j(x_j)$ and $\varphi_j(x_i) = w_i \neq \varphi_i(x_i)$.
	So $|U \cap X| = k - 2$. If we had $U \setminus A_{\ell} \neq \emptyset$, then $U$ would satisfy Condition (ii) above, which in turn would imply that $\{\varphi_i,\varphi_j\} \in E(\mathcal{G})$, thus contradicting the fact that $\mathcal{F}^*(u) \subseteq \mathcal{F}^*$ is an independent set in $\mathcal{G}$. So we see that
	$U \subseteq A_{\ell}$ and therefore
	$U = A_{\ell} \setminus \{x_i,x_j\}$. This in turn is equivalent to having
	$\text{Im}(\varphi_i) \cap \text{Im}(\varphi_j) =  \{w_1,\dots,w_k,u_0,\dots,u_{\ell}\} \setminus \{w_i,w_j\}$, as required.  
	
	Let us now show that
	$\text{Im}(\varphi_i) \cap \varphi_0(V(G^{\ell})) = \{w_1,\dots,w_k\} \setminus \{w_i\}$ holds for every $1 \leq i \leq k$.
	Fixing $1 \leq i \leq k$, set $U := U(\varphi_i,\varphi_0)$ and note that 
	$A_{\ell} \setminus \{x_i,y_{\ell}\} =
	\{x_1,\dots,x_k,y_0,\dots,y_{\ell-1}\} \setminus \{x_i\} \subseteq U$. Now, if
	$U \cap V(G^{\ell})$ were not contained in $X$, then $U$ would satisfy Condition (iii) above, which would imply that $\{\varphi_i,\varphi_0\} \in E(\mathcal{G})$, a contradiction. So we see that
	$U \cap V(G^{\ell}) \subseteq X$. Moreover, $x_i,y_{\ell} \notin U$, 
	because $\varphi_0(x_i) = w_i \neq \varphi_i(x_i)$ and $\varphi_i(y_{\ell}) = u_{\ell} \neq \varphi_0(y_{\ell})$. 
	So we see that
	$U \cap V(G^{\ell}) = \{x_1,\dots,x_k\} \setminus \{x_i\}$, which implies that
	$\text{Im}(\varphi_i) \cap \varphi_0(V(G^{\ell})) = 
	\{w_1,\dots,w_k\} \setminus \nolinebreak \{w_i\}$.
%
\end{proof}

Finally, we use Claim \ref{claim:Sarkozy_Selkow_main} in order to establish Item 3 of Lemma \ref{lem:Sarkozy_Selkow} by induction on $\ell$. The case $\ell = 0$ is trivial. Let us now fix $\ell \geq 0$, assume the validity of Item 3 of Lemma \ref{lem:Sarkozy_Selkow} for $\ell$ and prove its validity for $\ell + 1$. It is easy to see that if the assertion of Item 3(a) holds for parameter $\ell$, then it also holds for parameter $\ell + 1$. So we may assume that the assertion of Item 3(b) holds, namely, that $H$ contains at least $\delta' \cdot n^{k + \ell}$ copies of $G_{\ell}$, where $\delta' := \delta_{\ref{lem:Sarkozy_Selkow}}(\ell,r,\varepsilon)$. Now, apply Claim \ref{claim:Sarkozy_Selkow_main} to $H$ (with parameter $\delta'$ in place of $\varepsilon$). If Item 1 of Claim \ref{claim:Sarkozy_Selkow_main} holds, then Item 3(a) of Lemma \ref{lem:Sarkozy_Selkow} holds with $\ell + 1$ in place of $\ell$ (and with $j = \ell$). If instead Item 2 of Claim \ref{claim:Sarkozy_Selkow_main} holds, then $H$ contains at least $\delta \cdot n^{k + \ell + 1}$ copies of $G_{\ell+1}$, where $\delta = \delta_{\ref{claim:Sarkozy_Selkow_main}}(\ell,r,\delta')$. So we see that
Item 3(b) in Lemma \ref{lem:Sarkozy_Selkow} is satisfied. This completes the proof of the lemma. 
	
\section{An Improved Bound for a Problem of Erd\H{o}s and Gy\'{a}rf\'{a}s}\label{sec:gen_Ramsey}
The Brown--Erd\H{o}s--S\'{o}s problem has a known connection to (a special case of) the following generalized Ramsey problem, introduced by Erd\H{o}s and Gy\'{a}rf\'{a}s in \cite{EG}. Let $g(n,p,q)$ denote the minimum number of colors in a coloring of the edges of $K_n$ in which every copy of $K_p$ receives at least $q$ colors. For a fixed $p \geq 4$, Erd\H{o}s and Gy\'{a}rf\'{a}s \cite{EG} showed that $g(n,p,q)$ is quadratic in $n$ if and only if $q \geq q_{\text{quad}}(p) := \binom{p}{2} - \lfloor \frac{p}{2} \rfloor + 2$ and that, for $p$ even, $g(n,p,q_{\text{quad}}(p)) \leq \binom{n}{2} - \varepsilon n^2$ for some $\varepsilon = \varepsilon(p) > 0$. They then asked for which $q_{\text{quad}}(p) \leq q \leq \binom{p}{2}$ it holds that $g(n,p,q) = \binom{n}{2} - o(n^2)$, observing that this question is related to the Brown--Erd\H{o}s--S\'{o}s problem and using this relationship to prove some initial results. This connection was further exploited by S\'{a}rk\"{o}zy and Selkow, who combined it with \eqref{eq:Sarkozy_Selkow} (or, more precisely, with a $4$-uniform analogue thereof) to show that 
$g(n,p,q) = \binom{n}{2} - o(n^2)$ whenever 
$q > q_{\text{quad}}(p) + \lceil \frac{\log_2 p}{2} \rceil.$
By using our improved bound for the Brown--Erd\H{o}s--S\'{o}s problem (i.e., Corollary \ref{cor:general_BES}), we can improve upon the result of S\'{a}rk\"{o}zy and Selkow \cite{Sarkozy_Selkow_Ramsey}. For completeness, we now sketch the proof of the reduction from the above generalized Ramsey problem to the Brown--Erd\H{o}s--S\'{o}s problem. This reduction has been used implicitly in \cite{EG,Sarkozy_Selkow_Ramsey}. 

\begin{proposition}\label{prop:generalized_Ramsey}
	Let $p \geq 4$ and $q_{\text{quad}}(p) \leq q \leq \binom{p}{2}$. Set 
	$e := \binom{p}{2} - q + 1$. 
	If $f_4(n,p,e) = o(n^2)$, then $g(n,p,q) = \binom{n}{2} - o(n^2)$. 
\end{proposition}
\begin{proof}
	Assume that $f_4(n,p,e) = o(n^2)$ and suppose, for the sake of contradiction, that (for infinitely many $n$) there is a coloring of the edges of $K_n$ with $t := \binom{n}{2} - \varepsilon n^2$ colors (where $\varepsilon > 0$ is fixed) in which every copy of $K_p$ receives at least $q$ colors. 
	Then at least $\varepsilon n^2$ edges have the same color as some other edge, meaning that at least $\varepsilon n^2$ colors appear on more than one edge.
	
	Observe that each color appears fewer than $\lfloor \frac{p}{2} \rfloor$ times. Indeed, otherwise take edges $e_1,\dots,e_{\lfloor \frac{p}{2} \rfloor}$, all having the same color, and supplement them with (a suitable number of) vertices to obtain a copy of $K_p$ which receives at most 
	$\binom{p}{2} - \lfloor \frac{p}{2} \rfloor + 1 < q_{\text{quad}}(p) \leq q$ colors, a contradiction. It follows that at least $\varepsilon n^2/\lfloor p/2 \rfloor \geq 2\varepsilon n^2/p$ colors appear at least twice. 
	For each such color $c$, fix a pair of distinct edges $(e_1^c,e_2^c)$ which are colored with $c$. We claim that there are fewer than $(p-1)n/2$ colors $c$ for which $e_1^c$ and $e_2^c$ intersect. Indeed, assign to each such intersecting pair of edges their common vertex. If the number of intersecting pairs is at least $(p-1)n/2$, then there is a vertex $u$ which is the common vertex for at least $\lfloor \frac{p-1}{2} \rfloor$ such edge-pairs. 
	In other words, there are distinct vertices $(x_i,y_i : 1 \leq i \leq \lfloor \frac{p-1}{2} \rfloor)$
	such that the color of $\{u,x_i\}$ is the same as that of $\{u,y_i\}$ for each $1 \leq i \leq \lfloor \frac{p-1}{2} \rfloor$. As before, by adding a suitable number of vertices one obtains a copy of $K_p$ which receives at most 
	$\binom{p}{2} - \lfloor \frac{p-1}{2} \rfloor < q_{\text{quad}}(p) \leq q$ colors, contradicting our assumption.
	
	It follows from the previous two paragraphs that there are at least $2\varepsilon n^2/p - (p-1)n/2 \geq \varepsilon n^2/p$ colors $c$ (appearing at least twice) for which $e_1^c,e_2^c$ are disjoint. 
	Define an auxiliary $4$-graph $H$ on $V(K_n)$ by putting a ($4$-uniform) edge on $e_1^c \cup e_2^c$ for each color $c$ such that $e_1^c,e_2^c$ are disjoint. Then $e(H) \geq \varepsilon n^2/(3p)$ because $K_4$ has $3$ perfect matchings. 
	Observe, crucially, that $H$ contains no $(p,e)$-configuration. Indeed, if $H$ contained a $(p,e)$-configuration, then, by the definition of $H$ and our choice of $e$, the vertex set of this configuration would correspond to a copy of $K_p$ receiving at most
	$\binom{p}{2} - e = q-1$ colors, which is impossible. We thus conclude that $e(H) \leq f_4(n,p,e)$. On the other hand, $e(H) \geq \varepsilon n^2/(3p)$, implying that $f_4(n,p,e) = \Omega(n^2)$, contradicting our assumption. 
\end{proof}
	
	By Corollary \ref{cor:general_BES}, applied with parameters $r = 4$, $k = 2$ and $e = \binom{p}{2} - q + 1$, the bound $f_4(n,p,e) = o(n^2)$ holds whenever $p \geq 2e + \lceil 26\log e/\log\log e \rceil = 2(\binom{p}{2} - q + 1) + \lceil 26\log e/\log\log e \rceil$. By rearranging, we get the inequality 
	$q \geq \binom{p}{2} - \frac{p}{2} + 1 + \frac{1}{2} \cdot \lceil 26\log e/\log\log e \rceil$. Recalling the value of $q_{\text{quad}}(p)$ and using the (obvious) fact that $e \leq \binom{p}{2}$, we see that this inequality holds whenever 
	$q \geq  q_{\text{quad}}(p) + C\log p/\log\log p$ for some suitable absolute constant $C$. By combining this with Proposition \ref{prop:generalized_Ramsey}, we obtain the following improvement upon the aforementioned result from \cite{Sarkozy_Selkow_Ramsey}. 
	\begin{theorem}\label{theorem:generalized_Ramsey}
		There is an absolute constant $C$ such that $g(n,p,q) = \binom{n}{2} - o(n^2)$ for every $p \geq 4$ and $q \geq q_{\text{quad}}(p) + C\log p/\log\log p$.
	\end{theorem}

\end{document}